\documentclass[11pt]{amsart}

\usepackage{amsmath, amssymb, amsfonts, amsthm}
\usepackage{fullpage}
\usepackage{xypic}

\theoremstyle{plain}
\newtheorem{theorem}[subsection]{Theorem}
\newtheorem{proposition}[subsection]{Proposition}
\newtheorem{lemma}[subsection]{Lemma}

\theoremstyle{definition}
\newtheorem{definition}[subsection]{Definition}
\newtheorem{example}[subsection]{Example}
\newtheorem{examples}[subsection]{Examples}
\newtheorem{non-examples}[subsection]{Non-examples}

\newtheorem{remark}[subsection]{Remark}

\DeclareMathOperator{\colim}{colim}
\DeclareMathOperator{\Hom}{Hom}
\DeclareMathOperator{\coker}{coker}

\begin{document}

\title{Faithfully flat descent for projectivity of modules}
\author{Alexander Perry}
\maketitle

\section{Introduction}
\noindent
In this document we prove, following Raynaud and Gruson \cite{RG}, that the projectivity of modules descends along faithfully flat ring maps.  The idea of the proof is to use d\'evissage \`{a} la Kaplansky \cite{K} to reduce to the case of countably generated modules.  Given a well-behaved filtration of a module $M$, d\'evissage allows us to express $M$ as a direct sum of successive quotients of the filtering submodules (see Section \ref{section-devissage}).  Using this technique, we prove that a projective module is a direct sum of countably generated modules (Theorem \ref{theorem-projective-direct-sum}).  To prove descent of projectivity for countably generated modules, we introduce a ``Mittag-Leffler'' condition on modules, prove that a countably generated module is projective if and only if it is flat and Mittag-Leffler (Theorem \ref{theorem-projectivity-characterization}), and then show that the property of being a flat Mittag-Leffler module descends (Lemma \ref{lemma-ffdescent-flat-ML}).  Finally, given an arbitrary module $M$ whose base change by a faithfully flat ring map is projective, we filter $M$ by submodules whose successive quotients are countably generated projective modules, and then by d\'evissage conclude $M$ is a direct sum of projectives, hence projective itself (Theorem \ref{theorem-ffdescent-projectivity}).

One reason for the existence of this document is that there is an error in the proof of faithfully flat descent of projectivity in \cite{RG}.  There, descent of projectivity along faithfully flat ring maps is deduced from descent of projectivity along a more general type of ring map (\cite[Example 3.1.4(1) of Part II]{RG}).  However, the proof of descent along this more general type of map is incorrect, as explained in \cite{G}.  To patch this hole in the proof of faithfully flat descent of projectivity comes down to proving that the property of being a flat Mittag-Leffler module descends along faithfully flat ring maps.  We do this in Lemma \ref{lemma-ffdescent-flat-ML}.

\section{Notation and conventions}
\noindent
A ring is commutative with $1$.  We use $R$ and $S$ to denote rings, and $M$ and $N$ to denote modules.  

\section{Flat modules and universally injective module maps}

\noindent
In this section we first discuss criteria for flatness.  The main result in relation to this paper is Lazard's theorem (Theorem \ref{theorem-lazard} below), which says that a flat module is the colimit of a directed system of free finite modules.  Next we discuss universally injective module maps, which are in a sense complementary to flat modules (see Lemma \ref{lemma-flat-universally-injective}).  We follow Lazard's thesis \cite{Laz}; also see \cite{Lam}.

\subsection{Criteria for flatness}
\begin{definition}
Let $M$ be an $R$-module.  A relation $\sum_i a_i x_i = 0$ in $M$ where $x_i \in M, a_i \in R$ $(i =1, \dots n)$ is called \emph{trivial} if there exists $y_j \in M$ $(j = 1, \dots, m)$ and $b_{ij} \in R$ $(i = 1, \dots, n, j = 1, \dots, m)$ such that for all $i$
\[  x_i = \sum_{j} b_{ij} y_j \]
and for all $j$
\[  0   = \sum_{i} a_i b_{ij} .\]
Informally, a trivial relation in $M$ ``comes from'' relations in $R$.
\end{definition}

\begin{lemma}[Equational criterion for flatness]
\label{lemma-flat-eq}
Let $M$ be an $R$-module.  Then $M$ is flat if and only if every relation in $M$ is trivial. 
\end{lemma}

\begin{proof}
This is \cite[Lemma 7.28.8]{stacks-project}.
\end{proof}

\begin{lemma}
\label{lemma-flat-factors-free}
Let $M$ be an $R$-module.  The following are equivalent:
\begin{enumerate}
\item $M$ is flat.
\item If $f: R^n \rightarrow M$ is a module map and $x \in \ker(f)$, then there are module maps $h: R^n \rightarrow R^m$ and $g: R^m \rightarrow M$ such that $f = g \circ h$ and $x \in \ker(h)$.
\item Suppose $f: R^n \rightarrow M$ is a module map, $N \subset \ker(f)$ any submodule, and $h: R^n \rightarrow R^{m}$ a map such that $N \subset \ker(h)$ and $f$ factors through $h$.  Then given any $x \in \ker(f)$ we can find a map $h': R^n \rightarrow R^{m'}$ such that $N + Rx \subset \ker(h')$ and $f$ factors through $h'$. 
\item If $f: R^n \rightarrow M$ is a module map and $N \subset \ker(f)$ is a finitely generated submodule, then there are module maps $h: R^n \rightarrow R^m$ and $g: R^m \rightarrow M$ such that $f = g \circ h$ and $N \subset \ker(h)$.
\end{enumerate}
\end{lemma}

\begin{proof}
That (1) is equivalent to (2) is just a reformulation of the equational criterion for flatness.   To show (2) implies (3), let $g: R^m \rightarrow M$ be the map such that $f$ factors as $f = g \circ h$.  By (2) find $h'': R^m \rightarrow R^{m'}$ such that $h''$ kills $h(x)$ and $g: R^m \rightarrow M$ factors through $h''$.  Then taking $h' = h'' \circ h$ works.  (3) implies (4) by induction on the number of generators of $N \subset \ker(f)$ in (4).  Clearly (4) implies (2).
\end{proof}

\begin{lemma}
\label{lemma-flat-factors-fp}
Let $M$ be an $R$-module.  Then $M$ is flat if and only if the following condition holds: if $P$ is a finitely presented $R$-module and $f: P \rightarrow M$ a module map, then there is a free finite $R$-module $F$ and module maps $h: P \rightarrow F$ and $g: F \rightarrow M$ such that $f = g \circ h$.
\end{lemma}

\begin{proof}
This is just a reformulation of condition (4) from the previous lemma.
\end{proof}

\begin{lemma}
\label{lemma-flat-surjective-hom}
Let $M$ be an $R$-module.  Then $M$ is flat if and only if the following condition holds: for every finitely presented $R$-module $P$, if $N \rightarrow M$ is a surjective $R$-module map, then the induced map $\Hom_R(P,N) \rightarrow \Hom_R(P,M)$ is surjective.
\end{lemma}

\begin{proof}
First suppose $M$ is flat.  We must show that if $P$ is finitely presented, then given a map $f: P \rightarrow M$, it factors through the map $N \rightarrow M$.  By the previous lemma $f$ factors through a map $F \rightarrow M$ where $F$ is free and finite.  Since $F$ is free, this map factors through $N \rightarrow M$.  Thus $f$ factors through $N \rightarrow M$.

Conversely, suppose the condition of the lemma holds.  Let $f: P \rightarrow M$ be a map from a finitely presented module $P$.  Choose a free module $N$ with a surjection $N \rightarrow M$ onto $M$.  Then $f$ factors through $N \rightarrow M$, and since $P$ is finitely generated, $f$ factors through a free finite submodule of $N$.  Thus $M$ satisfies the condition of Lemma \ref{lemma-flat-factors-fp}, hence is flat.
\end{proof}

\begin{theorem}[Lazard's theorem]
\label{theorem-lazard}
Let $M$ be an $R$-module.  Then $M$ is flat if and only if it is the colimit of a directed system of free finite $R$-modules.
\end{theorem}

\begin{proof}
A colimit of a directed system of flat modules is flat, as taking directed colimits is exact and commutes with tensor product. Hence if $M$ is the colimit of a directed system of free finite modules then $M$ is flat.

For the converse, first recall that any module $M$ can be written as the colimit of a directed system of finitely presented modules, in the following way.  Choose a surjection $f: R^I \rightarrow M$ for some set $I$, and let $K$ be the kernel.  Let $E$ be the set of ordered pairs $(J,N)$ where $J$ is a finite subset of $I$ and $N$ is a finitely generated submodule of $R^J \cap K$.  Then $E$ is made into a directed partially ordered set by defining $(J,N) \leq (J',N')$ if and only if $J \subset J'$ and $N \subset N'$.  Define $M_{e} = R^J/N$ for $e = (J,N)$, and define $f_{ee'}: M_{e} \rightarrow M_{e'}$ to be the natural map for $e \leq e'$.  Then $(M_e,f_{ee'})$ is a directed system and the natural maps $f_e: M_{e} \rightarrow M$ induce an isomorphism $\colim_{e \in E} M_{e} \xrightarrow{\cong} M$.

Now suppose $M$ is flat.  Let $I = M \times \mathbf{Z}$, write $(x_{i})$ for the canonical basis of $R^{I}$, and take in the above discussion $f: R^I \rightarrow M$ to be the map sending $x_i$ to the projection of $i$ onto $M$.  To prove the theorem it suffices to show that the $e \in E$ such that $M_{e}$ is free form a cofinal subset of $E$.  So let $e = (J,N) \in E$ be arbitrary.  By Lemma \ref{lemma-flat-factors-fp} there is a free finite module $F$ and maps $h: R^J/N \rightarrow F$ and $g: F \rightarrow M$ such that the natural map $f_{e}: R^J/N \rightarrow M$ factors as $R^J/N \xrightarrow{h} F \xrightarrow{g} M$.  We are going to realize $F$ as $M_{e'}$ for some $e' \geq e$.    

Let $\{ b_1, \dots, b_n \}$ be a finite basis of $F$.  Choose $n$ distinct elements $i_1, \dots, i_n \in I$ such that $i_{\ell} \notin J$ for all $\ell$, and such that the image of $x_{i_{\ell}}$ under $f: R^I \rightarrow M$ equals the image of $b_{\ell}$ under $g: F \rightarrow M$.  This is possible by our choice of $I$.  Now let $J' = J \cup \{i_1, \dots , i_n \}$, and define $R^{J'} \rightarrow F$ by $x_i \mapsto h(x_i)$ for $i \in J$ and $x_{i_{\ell}} \mapsto b_{\ell}$ for $\ell = 1, \dots, n$.  Let $N' = \ker(R^{J'} \rightarrow F)$.  Observe:
\begin{enumerate}
\item $R^{J'} \rightarrow F$ factors $f: R^I \rightarrow M$, hence $N' \subset K = \ker(f)$;
\item $R^{J'} \rightarrow F$ is a surjection onto a free finite module, hence it splits and so $N'$ is finitely generated;
\item $J \subset J'$ and $N \subset N'$.
\end{enumerate}
By (1) and (2) $e' = (J',N')$ is in $E$, by (3) $e' \geq e$, and by construction $M_{e'} = R^{J'}/N' \cong F$ is free.
\end{proof}

\subsection{Universally injective module maps}
\begin{definition}
Let $f: M \rightarrow N$ be a map of $R$-modules.  Then $f$ is called \emph{universally injective} if for every $R$-module $Q$, the map $f \otimes_{R} \textnormal{id}_{Q}: M \otimes_{R} Q \rightarrow N \otimes_{R} Q$ is injective.  A sequence $0 \rightarrow M_1 \rightarrow M_2 \rightarrow M_3 \rightarrow 0$ of $R$-modules is called \emph{universally exact} if it is exact and $M_1 \rightarrow M_2$ is universally injective.  
\end{definition}

\begin{examples}
\label{examples-universally-exact}
\begin{enumerate} 
\item A split short exact sequence is universally exact since tensoring commutes with taking direct sums.
\item The colimit of a directed system of universally exact sequences is universally exact.  This follows from the fact that taking directed colimits is exact and that tensoring commutes with taking colimits.  In particular the colimit of a directed system of split exact sequences is universally exact.  We will see below that, conversely, any universally exact sequence arises in this way.
\end{enumerate}
\end{examples}

\noindent
Next we give a list of criteria for a short exact sequence to be universally exact.  They are analogues of criteria for flatness given above.  Parts (3)-(6) below correspond, respectively, to the criteria for flatness given in Lemma \ref{lemma-flat-eq}, Lemma \ref{lemma-flat-factors-free}, Lemma \ref{lemma-flat-surjective-hom}, and Theorem \ref{theorem-lazard}.
\begin{theorem}
\label{theorem-universally-exact-criteria}
Let 
\[ 0 \rightarrow M_1 \xrightarrow{f_1} M_2 \xrightarrow{f_2} M_3 \rightarrow 0 \]
be an exact sequence of $R$-modules.  The following are equivalent: 
\begin{enumerate}
\item The sequence $0 \rightarrow M_1 \rightarrow M_2 \rightarrow M_3 \rightarrow 0$ is universally exact.
\item For every finitely presented $R$-module $Q$, the sequence 
\[ 0 \rightarrow M_1 \otimes_{R} Q \rightarrow M_2 \otimes_{R} Q \rightarrow M_3 \otimes_{R} Q \rightarrow 0 \]
is exact.
\item Given elements $x_i \in M_1$ $(i = 1, \dots, n)$, $y_j \in M_2$ $(j = 1, \dots, m)$, and $a_{ij} \in R$ $(i = 1, \dots, n, j = 1, \dots, m)$ such that for all $i$
\[ f_1(x_i) = \sum_j a_{ij} y_j,\]
there exists $z_j \in M_1$ $(j =1, \dots, m)$ such that for all $i$,
\[ x_i = \sum_j a_{ij} z_j . \]
\item Given a commutative diagram of $R$-module maps
\[
\xymatrix{
R^n \ar[r] \ar[d] &  R^m \ar[d] \\
M_1 \ar[r]^{f_1}        &  M_2 
}
\]
where $m$ and $n$ are integers, there exists a map $R^m \rightarrow M_1$ making the top triangle commute.
\item For every finitely presented $R$-module $P$, the map $\Hom_R(P,M_2) \rightarrow \Hom_R(P,M_3)$ is surjective.
\item The sequence $0 \rightarrow M_1 \rightarrow M_2 \rightarrow M_3 \rightarrow 0$ is the colimit of a directed system of split exact sequences of the form
\[ 0 \rightarrow M_{1} \rightarrow M_{2,i} \rightarrow M_{3,i} \rightarrow 0  \]
where the $M_{3,i}$ are finitely presented.
\end{enumerate}
\end{theorem}

\begin{proof}
Obviously (1) implies (2).

Next we show (2) implies (3).  Let $f_1(x_i) = \sum_j a_{ij} y_j$ be relations as in (3).  Let $(f_j)$ be a basis for $R^m$, $(e_i)$ a basis for $R^n$, and $R^m \rightarrow R^n$ the map given by $f_{j} \mapsto \sum_{i} a_{ij} e_i$.  Let $Q$ be the cokernel of $R^m \rightarrow R^n$.  Then tensoring $R^m \rightarrow R^n \rightarrow Q \rightarrow 0$ by the map $f_1: M_1 \rightarrow M_2$, we get a commutative diagram
\[
\xymatrix{
M_1^{\oplus m} \ar[r] \ar[d] & M_1^{\oplus n} \ar[r] \ar[d] & M_1 \otimes_R Q \ar[r] \ar[d] & 0 \\
M_2^{\oplus m} \ar[r] & M_2^{\oplus n} \ar[r] & M_2 \otimes_R Q \ar[r] & 0
}
\]
where $M_1^{\oplus m} \rightarrow M_1^{\oplus n}$ is given by 
\[ \textstyle (z_1, \dots, z_m) \mapsto (\sum_{j} a_{1j} z_j, \dots, \sum_{j} a_{nj} z_j),\]
and $M_2^{\oplus m} \rightarrow M_2^{\oplus n}$ is given similarly.  We want to show $x = (x_1, \dots, x_n) \in M_1^{\oplus n}$ is in the image of $M_1^{\oplus m} \rightarrow M_1^{\oplus n}$.  By (2) the map $M_1 \otimes Q \rightarrow M_2 \otimes Q$ is injective, hence by exactness of the top row it is enough to show $x$ maps to $0$ in $M_2 \otimes Q$, and so by exactness of the bottom row it is enough to show the image of $x$ in $M_2^{\oplus n}$ is in the image of $M_2^{\oplus m} \rightarrow M_2^{\oplus n}$.  This is true by assumption.

Condition (4) is just a translation of (3) into diagram form.

Next we show (4) implies (5). Let $\varphi: P \rightarrow M_3$ be a map from a finitely presented $R$-module $P$.  We must show that $\varphi$ lifts to a map $P \rightarrow M_2$.  Choose a presentation of $P$, 
\[ R^n \xrightarrow{g_1} R^m \xrightarrow{g_2} P \rightarrow 0 . \]
Using freeness of $R^n$ and $R^m$, we can construct $h_2: R^m \rightarrow M_2$ and then $h_1: R^n \rightarrow M_1$ such that the following diagram commutes
\[
\xymatrix{
         & R^n \ar[r]^{g_1} \ar[d]^{h_1} & R^m \ar[r]^{g_2} \ar[d]^{h_2} & P \ar[r] \ar[d]^{\varphi} & 0 \\
0 \ar[r] & M_1 \ar[r]^{f_1} & M_2 \ar[r]^{f_2} & M_3 \ar[r] & 0 .
}
\]
By (4) there is a map $k_1: R^m \rightarrow M_1$ such that $k_1 \circ g_1 = h_1$.  Now define $h'_2: R^m \rightarrow M_2$ by $h_2' = h_2 - f_1 \circ k_1$.  Then
\[ h'_2 \circ g_1 = h_2 \circ g_1 - f_1 \circ k_1 \circ g_1 = h_2 \circ g_1 - f_1 \circ h_1 = 0 .
\]
Hence by passing to the quotient $h'_2$ defines a map $\varphi': P \rightarrow M_2$ such that $\varphi' \circ g_2 = h_2'$.  In a diagram, we have
\[
\xymatrix{
R^m \ar[r]^{g_2} \ar[d]_{h'_2} & P \ar[d]^{\varphi} \ar[dl]_{\varphi'} \\
M_2 \ar[r]^{f_2} & M_3.
}
\]
where the top triangle commutes.  We claim that $\varphi'$ is the desired lift, i.e.\ that $f_2 \circ \varphi' = \varphi$.  From the definitions we have
\[ f_2 \circ \varphi' \circ g_2 = f_2 \circ h'_2 = f_2 \circ h_2 - f_2 \circ f_1 \circ k_1 = f_2 \circ h_2 = \varphi \circ g_2 .\]
Since $g_2$ is surjective, this finishes the proof.

Now we show (5) implies (6).  Write $M_{3}$ as the colimit of a directed system of finitely presented modules $M_{3,i}$.  Let $M_{2,i}$ be the fiber product of $M_{3,i}$ and $M_{2}$ over $M_{3}$---by definition this is the submodule of $M_2 \times M_{3,i}$ consisting of elements whose two projections onto $M_3$ are equal.  Let $M_{1,i}$ be the kernel of the projection $M_{2,i} \rightarrow M_{3,i}$.  Then we have a directed system of exact sequences
\[ 0 \rightarrow M_{1,i} \rightarrow M_{2,i} \rightarrow M_{3,i} \rightarrow 0, \]
and for each $i$ a map of exact sequences
\[
\xymatrix{
0  \ar[r] & M_{1,i} \ar[d] \ar[r]  &  M_{2,i}  \ar[r] \ar[d] & M_{3,i} \ar[d] \ar[r] & 0 \\
0  \ar[r] & M_{1} \ar[r]  &  M_{2}  \ar[r] & M_{3} \ar[r] & 0
}
\] 
compatible with the directed system.  From the definition of the fiber product $M_{2,i}$, it follows that the map $M_{1,i} \rightarrow M_1$ is an isomorphism.  By (5) there is a map $M_{3,i} \rightarrow M_{2}$ lifting $M_{3,i} \rightarrow M_3$, and by the universal property of the fiber product this gives rise to a section of $M_{2,i} \rightarrow M_{3,i}$.  Hence the sequences
\[ 0 \rightarrow M_{1,i} \rightarrow M_{2,i} \rightarrow M_{3,i} \rightarrow 0 \] split.  Passing to the colimit, we have a commutative diagram
\[
\xymatrix{
0  \ar[r] & \colim M_{1,i} \ar[d]^{\cong} \ar[r]  &  \colim M_{2,i}  \ar[r] \ar[d] & \colim M_{3,i} \ar[d]^{\cong} \ar[r] & 0 \\
0  \ar[r] & M_{1} \ar[r]  &  M_{2}  \ar[r] & M_{3} \ar[r] & 0
}
\]
with exact rows and outer vertical maps isomorphisms.  Hence $\colim M_{2,i} \rightarrow M_2$ is also an isomorphism and (6) holds.

Condition (6) implies (1) by Examples \ref{examples-universally-exact} (2).
\end{proof}

\noindent
The previous theorem shows that a universally exact sequence is always a colimit of split short exact sequences.  If the cokernel of a universally injective map is finitely presented, then in fact the map itself splits:
\begin{lemma}
\label{lemma-universally-exact-split}
Let
\[ 0 \rightarrow M_1 \rightarrow M_2 \rightarrow M_3 \rightarrow 0 \]
be an exact sequence of $R$-modules.  Suppose $M_3$ is of finite presentation.  Then
\[ 0 \rightarrow M_1 \rightarrow M_2 \rightarrow M_3 \rightarrow 0 \]
is universally exact if and only if it is split.
\end{lemma}

\begin{proof}
A split sequence is always universally exact.  Conversely, if the sequence is universally exact, then by Theorem \ref{theorem-universally-exact-criteria} (5) applied to $P = M_3$, the map $M_{2} \rightarrow M_{3}$ admits a section.
\end{proof}

\noindent
The following lemma shows how universally injective maps are complementary to flat modules. 
\begin{lemma}
\label{lemma-flat-universally-injective}
Let $M$ be an $R$-module.  Then $M$ is flat if and only if any exact sequence of $R$-modules
\[ 0 \rightarrow M_1 \rightarrow M_2 \rightarrow M \rightarrow 0 \]
is universally exact.
\end{lemma}

\begin{proof}
This follows from Lemma \ref{lemma-flat-surjective-hom} and Theorem \ref{theorem-universally-exact-criteria} (5).
\end{proof}

\begin{examples}
\begin{enumerate}
\item In spite of Lemma \ref{lemma-universally-exact-split}, it is possible to have a short exact sequence of $R$-modules
\[ 0 \rightarrow M_1 \rightarrow M_2 \rightarrow M_3 \rightarrow 0 \]
that is universally exact but non-split.  For instance, take $R = \mathbf{Z}$, let $M_1 = \bigoplus_{n=1}^{\infty} \mathbf{Z}$, let $M_{2} = \prod_{n = 1}^{\infty} \mathbf{Z}$, and let $M_{3}$ be the cokernel of the inclusion $M_1 \rightarrow M_2$.  Then $M_1,M_2,M_3$ are all flat since they are torsion-free, so by Lemma \ref{lemma-flat-universally-injective},
\[ 0 \rightarrow M_1 \rightarrow M_2 \rightarrow M_3 \rightarrow 0 \]
is universally exact.  However there can be no section $s: M_3 \rightarrow M_2$.  In fact, if $x$ is the image of $(2,2^2,2^3, \dots) \in M_2$ in $M_3$, then any module map $s: M_3 \rightarrow M_2$ must kill $x$.  This is because $x \in 2^n M_3$ for any $n \geq 1$, hence $s(x)$ is divisible by $2^n$ for all $n \geq 1$ and so must be $0$. 

\item In spite of Lemma \ref{lemma-flat-universally-injective}, it is possible to have a short exact sequence of $R$-modules
\[ 0 \rightarrow M_1 \rightarrow M_2 \rightarrow M_3 \rightarrow 0 \]
that is universally exact but with $M_1,M_2,M_3$ all non-flat.  In fact if $M$ is any non-flat module, just take the split exact sequence 
\[ 0 \rightarrow M \rightarrow M \oplus M \rightarrow M \rightarrow 0. \]
For instance over $R = \mathbf{Z}$, take $M$ to be any torsion module.

\item Taking the direct sum of an exact sequence as in (1) with one as in (2), we get a short exact sequence of $R$-modules
\[ 0 \rightarrow M_1 \rightarrow M_2 \rightarrow M_3 \rightarrow 0 \]
that is universally exact, non-split, and such that $M_1,M_2,M_3$ are all non-flat.
\end{enumerate}
\end{examples}

\noindent
We end this section with a simple observation.
\begin{lemma}
\label{lemma-ui-flat-domain}
Let $f: M_1 \rightarrow M_2$ be a universally injective map of $R$-modules, and suppose $M_2$ is flat.  Then $M_1$ is flat.
\end{lemma}

\begin{proof}
Let $N \rightarrow N'$ be an injection and consider the commutative diagram
\[
\xymatrix{
M_1 \otimes_{R} N \ar[r] \ar[d] & M_1 \otimes_{R} N' \ar[d] \\
M_2 \otimes_{R} N  \ar[r]       & M_2 \otimes_{R} N'.
}
\]
From the assumptions, the vertical arrows and the bottom horizontal arrow are all injective, hence so is the top arrow.  We conclude that $M_1$ is flat.
\end{proof}

\section{The case of finite projective modules}
\noindent
In this section we give an elementary proof of the fact that the property of being a \emph{finite} projective module descends along faithfully flat ring maps.  The proof does not apply when we drop the finiteness condition, and the result will not be used in the rest of the paper.  However, the method is indicative of the one we shall use to prove descent for the property of being a \emph{countably generated} projective module---see the comments at the end of this section.

\begin{lemma}
\label{lemma-finite-projective}
Let $M$ be an $R$-module.  Then $M$ is finite projective if and only if $M$ is finitely presented and flat.
\end{lemma}

\begin{proof}
First suppose $M$ is finite projective. Any projective module is flat, so we just need to show that $M$ is also finitely presented.  Take a surjection $R^n \rightarrow M$ and let $K$ be the kernel.  Since $M$ is projective, $0 \rightarrow K \rightarrow R^n \rightarrow M \rightarrow 0$ splits.  So $K$ is a direct summand of $R^n$ and thus finitely generated.  This shows $M$ is finitely presented.   

Conversely, if $M$ is finitely presented and flat, then take a surjection $R^n \rightarrow M$.  By Lemma \ref{lemma-flat-surjective-hom} applied to $P = M$, the map $R^n \rightarrow M$ admits a section.  So $M$ is a direct summand of a free module and hence projective.
\end{proof}

\begin{lemma}
\label{lemma-descend-properties-modules}
Let $R \to S$ be a faithfully flat ring map.
Let $M$ be an $R$-module. Then:
\begin{enumerate}
\item If the $S$-module $M \otimes_R S$ is of finite presentation, then
$M$ is of finite presentation.
\item If the $S$-module $M \otimes_R S$ is flat, then $M$ is flat.
\end{enumerate}
\end{lemma}

\begin{proof}
This is \cite[Lemma 7.114.8]{stacks-project}.
\end{proof}

\begin{proposition}
\label{proposition-ffdescent-finite-projectivity}
Let $R \rightarrow S$ be a faithfully flat ring map.  Let $M$ be an $R$-module.  If the $S$-module $M \otimes_{R} S$ is finite projective, then $M$ is finite projective. 
\end{proposition}

\begin{proof}
Follows from Lemmas \ref{lemma-finite-projective} and \ref{lemma-descend-properties-modules}.
\end{proof}

\noindent
The rest of this paper is about removing the finiteness assumption by using d\'evissage to reduce to the countably generated case.  In the countably generated case, the strategy is to find a characterization of countably generated projective modules analogous to Lemma \ref{lemma-finite-projective}, and then to prove directly that this characterization descends. We do this by introducing the notion of a Mittag-Leffer module and proving that if a module $M$ is countably generated, then it is projective if and only if it is flat and Mittag-Leffler (Theorem \ref{theorem-projectivity-characterization}).  When $M$ is finitely generated, this statement reduces to Lemma \ref{lemma-finite-projective} (since, according to Examples \ref{examples-ML} (1), a finitely generated module is Mittag-Leffler if and only if it is finitely presented).

\section{Transfinite D\'{e}vissage of modules \`{a} la Kaplansky} 
\label{section-devissage}
\noindent
In this section we introduce a d\'evissage technique for decomposing a module into a direct sum.  The main result is that a projective module is a direct sum of countably generated modules (Theorem \ref{theorem-projective-direct-sum} below).  We follow \cite{K}.

\subsection{Transfinite d\'evissage of modules}
\begin{definition}
\label{definition-devissage}
Let $M$ be an $R$-module.  A \emph{direct sum d\'{e}vissage} of $M$ is a family of submodules $(M_{\alpha})_{\alpha \in S}$, indexed by an ordinal $S$ and increasing (with respect to inclusion), such that:
\begin{enumerate}
\item[(0)] $M_0 = 0$;
\item[(1)] $M = \bigcup_{\alpha} M_{\alpha}$;
\item[(2)] if $\alpha \in S$ is a limit ordinal, then $M_{\alpha} = \bigcup_{\beta < \alpha} M_{\beta}$;
\item[(3)] if $\alpha + 1 \in S$, then $M_{\alpha}$ is a direct summand of $M_{\alpha+1}$.
\end{enumerate}
If moreover
\begin{enumerate}
\item[(4)] $M_{\alpha+1}/M_{\alpha}$ is countably generated for $\alpha+1 \in S$,
\end{enumerate}
then $(M_{\alpha})_{\alpha \in S}$ is called a \emph{Kaplansky d\'{e}vissage} of $M$.
\end{definition}

\noindent
The terminology is justified by the following lemma.
\begin{lemma}
\label{lemma-direct-sum-devissage}
Let $M$ be an $R$-module.  If $(M_{\alpha})_{\alpha \in S}$ is a direct sum d\'evissage of $M$, then $M \cong \bigoplus_{\alpha +1  \in S} M_{\alpha+1}/M_{\alpha}$.
\end{lemma}

\begin{proof}
By property (3) of a direct sum d\'evissage, there is an inclusion $M_{\alpha+1}/M_{\alpha} \rightarrow M$ for each $\alpha \in S$.  Consider the map
\[ f: \bigoplus_{\alpha + 1\in S} M_{\alpha+1}/M_{\alpha} \rightarrow M \]
given by the sum of these inclusions.  Transfinite induction on $S$ shows that the image contains $M_{\alpha}$ for every $\alpha \in S$: for $\alpha = 0$ this is true by (0); if $\alpha+1$ is a successor ordinal then it is clearly true; and if $\alpha$ is a limit ordinal and it is true for $\beta < \alpha$, then it is true for $\alpha$ by (2). Hence $f$ is surjective by (1).  

Transfinite induction on $S$ also shows that for every $\beta \in S$ the restriction
\[ f_{\beta} : \bigoplus_{\alpha + 1 \leq \beta} M_{\alpha+1}/M_{\alpha} \rightarrow M \] 
of $f$ is injective: For $\beta = 0$ it is true.  If it is true for all $\beta' < \beta$, then let $x$ be in the kernel and write $x = (x_{\alpha+1})_{\alpha+1 \leq \beta}$ in terms of its components $x_{\alpha+1} \in M_{\alpha+1}/M_{\alpha}$.  By property (3) both $(x_{\alpha+1})_{\alpha+1 < \beta}$ and $x_{\beta+1}$ map to $0$.  Hence $x_{\beta+1} = 0$ and, by the assumption that the restriction $f_{\beta'}$ is injective for all $\beta' < \beta$, also $x_{\alpha+1} = 0$ for every $\alpha+1 < \beta$.  So $x = 0$ and $f_{\beta}$ is injective, which finishes the induction.  We conclude that $f$ is injective since $f_{\beta}$ is for each $\beta \in S$.
\end{proof}

\begin{lemma}
\label{lemma-Kaplansky-devissage}
Let $M$ be an $R$-module.  Then $M$ is a direct sum of countably generated $R$-modules if and only if it admits a Kaplansky d\'evissage.
\end{lemma}

\begin{proof}
The lemma takes care of the ``if'' direction.  Conversely, suppose $M = \bigoplus_{i \in I} N_i$ where each $N_i$ is a countably generated $R$-module.  Well-order $I$ so that we can think of it as an ordinal.  Then setting $M_{i} = \bigoplus_{j < i} N_{j}$ gives a Kaplansky d\'evissage $(M_i)_{i \in I}$ of $M$.  
\end{proof}

\begin{theorem}
\label{theorem-kaplansky-direct-sum}
Suppose $M$ is a direct sum of countably generated $R$-modules.  If $P$ is a direct summand of $M$, then $P$ is also a direct sum of countably generated $R$-modules.
\end{theorem}

\begin{proof}
Write $M = P \oplus Q$.  We are going to construct a Kaplansky d\'evissage $(M_{\alpha})_{\alpha \in S}$ of $M$ which, in addition to the defining properties (0)-(4), satisfies:
\begin{enumerate}
\item[(5)] Each $M_{\alpha}$ is a direct summand of $M$;

\item[(6)] $M_{\alpha} = P_{\alpha} \oplus Q_{\alpha}$, where $P_{\alpha} =P \cap M_{\alpha}$ and $Q = Q \cap M_{\alpha}$.
\end{enumerate}
(Note: if properties (0)-(2) hold, then in fact property (3) is equivalent to property (5).)

To see how this implies the theorem, it is enough to show that $(P_{\alpha})_{\alpha \in S}$ forms a Kaplansky d\'evissage of $P$.  Properties (0), (1), and (2) are clear.  By (5) and (6) for $(M_{\alpha})$, each $P_{\alpha}$ is a direct summand of $M$.  Since $P_{\alpha} \subset P_{\alpha + 1}$, this implies $P_{\alpha}$ is a direct summand of $P_{\alpha + 1}$; hence (3) holds for $(P_{\alpha})$.  For (4), note that
\[ 
M_{\alpha +1}/M_{\alpha} \cong P_{\alpha+1}/P_{\alpha} \oplus Q_{\alpha+1}/Q_{\alpha},
\]
so $P_{\alpha+1}/P_{\alpha}$ is countably generated because this is true of $M_{\alpha+1}/M_{\alpha}$.

It remains to construct the $M_{\alpha}$.  Write $M = \bigoplus_{i \in I} N_i$ where each $N_i$ is a countably generated $R$-module.  Choose a well-ordering of $I$.  By transfinite induction we are going to define an increasing family of submodules $M_{\alpha}$ of $M$, one for each ordinal $\alpha$, such that $M_{\alpha}$ is a direct sum of some subset of the $N_i$.

For $\alpha = 0$ let $M_{0} = 0$.  If $\alpha$ is a limit ordinal and $M_{\beta}$ has been defined for all $\beta < \alpha$, then define $M_{\alpha} = \bigcup_{\beta < \alpha} M_{\beta}$.  Since each $M_{\beta}$ for $\beta < \alpha$ is a direct sum of a subset of the $N_i$, the same will be true of $M_{\alpha}$.  If $\alpha+1$ is a successor ordinal and $M_{\alpha}$ has been defined, then define $M_{\alpha+1}$ as follows.  If $M_{\alpha} = M$, then let $M_{\alpha+1} = M$.  If not, choose the smallest $j \in I$ such that $N_{j}$ is not contained in $M_{\alpha}$.  We will construct an infinite matrix $(x_{mn}), m,n = 1, 2, 3, \dots$ such that: 
\begin{enumerate}
\item $N_j$ is contained in the submodule of $M$ generated by the entries $x_{mn}$; 
\item if we write any entry $x_{k\ell}$ in terms of its $P$- and $Q$-components, $x_{k\ell} = y_{k\ell} + z_{k\ell}$, then the matrix $(x_{mn})$ contains a set of generators for each $N_i$ for which $y_{k\ell}$ or $z_{k\ell}$ has nonzero component.
\end{enumerate}
Then we define $M_{\alpha+1}$ to be the submodule of $M$ generated by $M_{\alpha}$ and all $x_{mn}$; by property (2) of the matrix $(x_{mn})$, $M_{\alpha+1}$ will be a direct sum of some subset of the $N_i$.  To construct the matrix $(x_{mn})$, let $x_{11},x_{12},x_{13}, \dots$ be a countable set of generators for $N_j$.  Then if $x_{11} = y_{11} + z_{11}$ is the decomposition into $P$- and $Q$-components, let $x_{21},x_{22},x_{23}, \dots$ be a countable set of generators for the sum of the $N_i$ for which $y_{11}$ or $z_{11}$ have nonzero component.  Repeat this process on $x_{12}$ to get elements $x_{31}, x_{32}, \dots$, the third row of our matrix.  Repeat on $x_{21}$ to get the fourth row, on $x_{13}$ to get the fifth, and so on, going down along successive anti-diagonals as indicated below:
\[
\left(
\vcenter{\xymatrix@R=2mm@C=2mm{
x_{11} & x_{12} \ar[dl] & x_{13} \ar[dl] & x_{14} \ar[dl] & \cdots  \\
x_{21} & x_{22} \ar[dl] & x_{23} \ar[dl] & \cdots  \\
x_{31} & x_{32} \ar[dl] & \cdots \\
x_{41} & \cdots \\
\cdots 
}}
\right).
\]

Transfinite induction on $I$ (using the fact that we constructed $M_{\alpha+1}$ to contain $N_{j}$ for the smallest $j$ such that $N_{j}$ is not contained in $M_{\alpha}$) shows that for each $i \in I$, $N_{i}$ is contained in some $M_{\alpha}$.  Thus, there is some large enough ordinal $S$ satisfying: for each $i \in I$ there is $\alpha \in S$ such that $N_{i}$ is contained in $M_{\alpha}$.  This means $(M_{\alpha})_{\alpha \in S}$ satisfies property (1) of a Kaplansky d\'evissage of $M$.  The family $(M_{\alpha})_{\alpha \in S}$ moreover satisfies the other defining properties, and also (5) and (6) above: properties (0), (2), (4), and (6) are clear by construction;  property (5) is true because each $M_{\alpha}$ is by construction a direct sum of some $N_{i}$; and (3) is implied by (5) and the fact that $M_{\alpha} \subset M_{\alpha+1}$.
\end{proof}

\noindent
As a corollary we get the result for projective modules stated at the beginning of the section.
\begin{theorem}
\label{theorem-projective-direct-sum}
If $P$ is a projective $R$-module, then $P$ is a direct sum of countably generated projective $R$-modules.
\end{theorem}

\begin{proof}
A module is projective if and only if it is a direct summand of a free module, so this follows from Theorem \ref{theorem-kaplansky-direct-sum}.
\end{proof}

\subsection{Projective modules over a local ring}
\noindent
In the remainder of this section we prove a result of independent interest from the rest of the paper: a projective module $M$ over a local ring is free (Theorem \ref{theorem-projective-free-over-local-ring} below).  Note that with the additional assumption that $M$ is finite, this result is elementary (\cite[Lemma 7.65.4]{stacks-project}).  In general, by Theorem \ref{theorem-projective-direct-sum}, we have:

\begin{lemma}
\label{lemma-projective-free}
Let $R$ be a ring.  Then every projective $R$-module is free if and only if every countably generated projective $R$-module is free.
\end{lemma}

\noindent
Here is a criterion for a countably generated module to be free.
\begin{lemma}
\label{lemma-freeness-criteria}
Let $M$ be a countably generated $R$-module.  Suppose any direct summand $N$ of $M$ satisfies: any element of $N$ is contained in a free direct summand of $N$.  Then $M$ is free.
\end{lemma}

\begin{proof}
Let $x_1, x_2, \dots$ be a countable set of generators for $M$.  By the assumption on $M$, we can construct by induction free $R$-modules $F_1,F_2, \dots$ such that for every positive integer $n$, $\bigoplus_{i=1}^{n} F_i$ is a direct summand of $M$ and contains $x_1, \dots, x_n$.  Then $M = \bigoplus_{i = 1}^{\infty} F_i$.
\end{proof}

\begin{lemma}
\label{lemma-projective-freeness-criteria}
Let $P$ be a projective module over a local ring $R$.  Then any element of $P$ is contained in a free direct summand of $P$.
\end{lemma}

\begin{proof}
Since $P$ is projective it is a direct summand of some free $R$-module $F$, say $F = P \oplus Q$.  Let $x \in P$ be the element that we wish to show is contained in a free direct summand of $P$.  Let $B$ be a basis of $F$ such that the number of basis elements needed in the expression of $x$ is minimal, say $x = \sum_{i=1}^n a_i e_i$ for some $e_i \in B$ and $a_i \in R$.  Then no $a_j$ can be expressed as a linear combination of the other $a_i$; for if $a_j = \sum_{i \neq  j} a_i b_i$ for some $b_i \in R$, then replacing $e_i$ by $e_i + b_ie_j$ for $i \neq j$ and leaving unchanged the other elements of $B$, we get a new basis for $F$ in terms of which $x$ has a shorter expression.

Let $e_i = y_i + z_i, y_i \in P, z_i \in Q$ be the decomposition of $e_i$ into its $P$- and $Q$-components.  Write $y_i = \sum_{j=1}^{n} b_{ij} e_j + t_i$, where $t_i$ is a linear combination of elements in $B$ other than $e_1, \dots, e_n$.  To finish the proof it suffices to show that the matrix $(b_{ij})$ is invertible.  For then the map $F \rightarrow F$ sending $e_i \mapsto y_i$ for $i=1, \dots, n$ and fixing $B \setminus \{e_1, \dots, e_n\}$ is an isomorphism, so that $y_1, \dots, y_n$ together with $B \setminus \{e_1, \dots, e_n\}$ form a basis for $F$.  Then the submodule $N$ spanned by $y_1, \dots, y_n$ is a free submodule of $P$; $N$ is a direct summand of $P$ since $N \subset P$ and both $N$ and $P$ are direct summands of $F$; and $x \in N$ since $x \in P$ implies $x = \sum_{i=1}^n a_i e_i = \sum_{i=1}^n a_i y_i$.

Now we prove that $(b_{ij})$ is invertible. Plugging $y_i = \sum_{j=1}^{n} b_{ij} e_j + t_i$ into $\sum_{i=1}^n a_i e_i = \sum_{i=1}^n a_i y_i$ and equating the coefficients of $e_j$ gives $a_j = \sum_{i=1}^n a_i b_{ij}$.  But as noted above, our choice of $B$ guarantees that no $a_j$ can be written as a linear combination of the other $a_i$.  Thus $b_{ij}$ is a non-unit for $i \neq j$, and $1-b_{ii}$ is a non-unit---so in particular $b_{ii}$ is a unit---for all $i$.  But a matrix over a local ring having units along the diagonal and non-units elsewhere is invertible, as its determinant is a unit.
\end{proof}

\begin{theorem}
\label{theorem-projective-free-over-local-ring}
If $P$ is a projective module over a local ring $R$, then $P$ is free.
\end{theorem}

\begin{proof}
Follows from Lemmas \ref{lemma-projective-free}, \ref{lemma-freeness-criteria}, and \ref{lemma-projective-freeness-criteria}.
\end{proof}

\section{Mittag-Leffler modules}
\noindent
The purpose of this section is to define Mittag-Leffler modules and to discuss their basic properties.
 
\subsection{Mittag-Leffler systems}
In the following, $I$ will be a directed partially ordered set.  Let $(A_i, \varphi_{ji}: A_j \rightarrow A_{i})$ be a directed inverse system of sets or of modules indexed by $I$.  For each $i \in I$, the images $\varphi_{ji}(A_j) \subset A_i$ for $j \geq i$ form a decreasing family.  Let $A'_i = \bigcap_{j \geq i} \varphi_{ji}(A_j)$.  Then $\varphi_{ji}(A'_j) \subset A'_i$ for $j \geq i$, hence by restricting we get a directed inverse system $(A'_{i}, \varphi_{ji}|_{A'_j})$.  From the construction of the limit of an inverse system in the category of sets or modules, we have $\lim A_i = \lim A'_i$.  The Mittag-Leffler condition on $(A_i, \varphi_{ji})$ is that $A'_i$ equals $\varphi_{ji}(A_{j})$ for some $j \geq i$ (and hence equals $\varphi_{ki}(A_k)$ for all $k \geq j$):

\begin{definition}
\label{definition-ML-system}
Let $(A_{i},\varphi_{ji})$ be a directed inverse system of sets over $I$.  Then we say  $(A_{i}, \varphi_{ji})$ is \emph{Mittag-Leffler inverse system} if for each $i \in I$, the decreasing family $\varphi_{ji}(A_{j}) \subset A_i$ for $j \geq i$ stabilizes.  Explicitly, this means that for each $i \in I$, there exists $j \geq i$ such that for $k \geq j$ we have $\varphi_{ki}(A_{k}) = \varphi_{ji}( A_{j})$.  If $(A_{i},\varphi_{ji})$ is a directed inverse system of modules over a ring $R$, we say that it is Mittag-Leffler if the underlying inverse system of sets is Mittag-Leffler.
\end{definition}

\begin{example}
\label{example-ML-surjective-maps}
If $(A_i, \varphi_{ji})$ is a directed inverse system of sets or of modules and the maps $\varphi_{ji}$ are surjective, then clearly the system is Mittag-Leffler.  Conversely, suppose $(A_i, \varphi_{ji})$ is Mittag-Leffler.  Let $A'_{i} \subset A_i$ be the stable image of $\varphi_{ji}(A_j)$ for $j \geq i$.  Then $\varphi_{ji}|_{A'_{j}}: A'_{j} \rightarrow A'_i$ is surjective for $j \geq i$ and $\lim A_i = \lim A'_i$.  Hence the limit of the Mittag-Leffler system $(A_i, \varphi_{ji})$ can also be written as the limit of a directed inverse system over $I$ with surjective maps.  
\end{example}

\begin{lemma}
\label{lemma-ML-limit-nonempty}
Let $(A_{i}, \varphi_{ji})$ be a directed inverse system over $I$.  Suppose $I$ is countable.  If $(A_{i}, \varphi_{ji})$ is Mittag-Leffler and the $A_i$ are nonempty, then $\lim A_i$ is nonempty.
\end{lemma}

\begin{proof}
Let $i_1,i_2, i_3, \dots$ be an enumeration of the elements of $I$.  Define inductively a sequence of elements $j_n \in I$ for $n = 1,2,3, \dots$ by the conditions: $j_1 = i_1$, and $j_n \geq i_n$ and $j_{n} \geq j_{m}$ for $m < n$.  Then the sequence $j_n$ is increasing and forms a cofinal subset of $I$.  Hence we may assume $I =\{1,2,3,\dots \}$.  So by Example \ref{example-ML-surjective-maps} we are reduced to showing that the limit of an inverse system of nonempty sets with surjective maps indexed by the positive integers is nonempty.  This is obvious.  
\end{proof}

\noindent
The Mittag-Leffler condition will be important for us because of the following exactness property.
\begin{lemma}
\label{lemma-ML-exact-sequence}
Let
\[ 0 \rightarrow A_i \xrightarrow{f_i} B_i \xrightarrow{g_i} C_i \rightarrow 0 \]
be an exact sequence of directed inverse systems of abelian groups over $I$.  Suppose $I$ is countable.  If $(A_i)$ is Mittag-Leffler, then
\[ 0 \rightarrow \lim A_i \rightarrow \lim B_i \rightarrow \lim C_i\rightarrow 0 \] 
is exact.
\end{lemma}

\begin{proof}
Taking limits of directed inverse systems is left exact, hence we only need to prove surjectivity of $\lim B_i \rightarrow \lim C_i$.  So let $(c_i) \in \lim C_i$.  For each $i \in I$, let $E_i = g_{i}^{-1}(c_i)$, which is nonempty since $g_i: B_i \rightarrow C_i$ is surjective. The system of maps $\varphi_{ji}: B_j \rightarrow B_i$ for $(B_i)$ restrict to maps $E_{j} \rightarrow E_{i}$ which make $(E_i)$ into an inverse system of nonempty sets.  It is enough to show that $(E_i)$ is Mittag-Leffler. For then Lemma \ref{lemma-ML-limit-nonempty} would show $\lim E_i$ is nonempty, and taking any element of $\lim E_i$ would give an element of $\lim B_i$ mapping to $(c_i)$.

By the injection $f_i: A_i \rightarrow B_i$ we will regard $A_i$ as a subset of $B_i$.  Since $(A_i)$ is Mittag-Leffler, if $i \in I$ then there exists $j \geq i$ such that $\varphi_{ki}(A_k) = \varphi_{ji}(A_j)$ for $k \geq j$.  We claim that also $\varphi_{ki}(E_k) = \varphi_{ji}(E_j)$ for $k \geq j$.  Always $\varphi_{ki}(E_k) \subset \varphi_{ji}(E_j)$ for $k \geq j$.  For the reverse inclusion let $e_j \in E_j$, and we need to find $x_k \in E_k$ such that $\varphi_{ki}(x_k) = \varphi_{ji}(e_j)$.  Let $e'_k \in E_k$ be any element, and set $e'_j = \varphi_{kj}(e'_k)$.  Then $g_j(e_j - e'_j) = c_j - c_j = 0$, hence $e_j - e'_j = a_j \in A_j$.  Since $\varphi_{ki}(A_k) = \varphi_{ji}(A_j)$, there exists $a_k \in A_k$ such that $\varphi_{ki}(a_k) = \varphi_{ji}(a_j)$.  Hence
\[ \varphi_{ki}(e'_k + a_k) = \varphi_{ji}(e'_j) + \varphi_{ji}(a_j) = \varphi_{ji}(e_j), \]
so we can take $x_k = e'_k + a_k$.
\end{proof}

\subsection{Mittag-Leffler modules}
\begin{definition}
\label{definition-ML-inductive-system}
Let $(M_i, f_{ij})$ be a directed system of $R$-modules.  We say that $(M_i,f_{ij})$ is a \emph{Mittag-Leffler directed system of modules} if each $M_i$ is of finite presentation and if for every $R$-module $N$, the inverse system $(\Hom_{R}(M_i,N),\Hom_{R}(f_{ij},N))$ is Mittag-Leffler.
\end{definition}

\noindent
We are going to characterize those $R$-modules that are colimits of Mittag-Leffler directed systems of modules.  

\begin{definition}
\label{definition-domination}
Let $f: M \rightarrow N$ and $g: M \rightarrow M'$ be maps of $R$-modules.  Then we say $g$ \emph{dominates} $f$ if for any $R$-module $Q$, we have $\ker(f \otimes_{R} \textnormal{id}_Q) \subset \ker(g \otimes_{R} \textnormal{id}_Q)$.
\end{definition}

\begin{lemma}
\label{lemma-domination-fp}
Let $f: M \rightarrow N$ and $g: M \rightarrow M'$ be maps of $R$-modules.  Then $g$ dominates $f$ if and only if for any finitely presented $R$-module $Q$, we have $\ker(f \otimes_{R} \textnormal{id}_Q) \subset \ker(g \otimes_{R} \textnormal{id}_Q)$. 
\end{lemma}

\begin{proof}
Suppose $\ker(f \otimes_{R} \textnormal{id}_Q) \subset \ker(g \otimes_{R} \textnormal{id}_Q)$ for all finitely presented modules $Q$.  If $Q$ is an arbitrary module, write $Q = \colim_{i \in I} Q_i$ as a colimit of a directed system of finitely presented modules $Q_i$.  Then $\ker(f \otimes_{R} \textnormal{id}_{Q_i}) \subset \ker(g \otimes_{R} \textnormal{id}_{Q_i})$ for all $i$.  Since taking directed colimits is exact and commutes with tensor product, it follows that $\ker(f \otimes_{R} \textnormal{id}_Q) \subset \ker(g \otimes_{R} \textnormal{id}_Q)$.
\end{proof}

\noindent
The above definition of domination is related to the usual notion of domination of maps as follows.
\begin{lemma}
\label{lemma-domination}
Let $f: M \rightarrow N$ and $g: M \rightarrow M'$ be maps of $R$-modules.  Suppose $\coker(f)$ is of finite presentation.  Then $g$ dominates $f$ if and only if $g$ factors through $f$, i.e.\ there exists a module map $h: N \rightarrow M'$ such that $g = h \circ f$.
\end{lemma}

\begin{proof}
Consider the pushout of $f$ and $g$,
\[ 
\xymatrix{
M  \ar[r]^{f} \ar[d]_{g} & N \ar[d]^{g'} \\
M' \ar[r]^{f'} & N'
}
\]
where $N'$ is $M' \oplus N$ modulo the submodule consisting of elements $(g(x),-f(x))$ for $x \in M$.  We are going to show that the two conditions we wish to prove equivalent are each equivalent to $f'$ being universally injective.

From the definition of $N'$ we have a short exact sequence
\[ 0 \rightarrow \ker(f) \cap \ker(g) \rightarrow \ker(f) \rightarrow \ker(f') \rightarrow 0. \]
Since tensoring commutes with taking pushouts, we have such a short exact sequence
\[ 0 \rightarrow \ker(f \otimes \textnormal{id}_{Q} ) \cap \ker(g \otimes \textnormal{id}_{Q}) \rightarrow \ker(f \otimes \textnormal{id}_{Q}) \rightarrow \ker(f' \otimes \textnormal{id}_{Q}) \rightarrow 0 \]
for every $R$-module $Q$.  So $f'$ is universally injective if and only if $\ker(f \otimes \textnormal{id}_{Q} ) \subset \ker(g \otimes \textnormal{id}_{Q})$ for every $Q$, if and only if $g$ dominates $f$.  

On the other hand, from the definition of the pushout it follows that $\coker(f') = \coker(f)$, so $\coker(f')$ is of finite presentation.  Then by Lemma \ref{lemma-universally-exact-split}, $f'$ is universally injective if and only if
\[
0 \rightarrow M' \xrightarrow{f'} N' \rightarrow \coker(f') \rightarrow 0
\]
splits.  This is the case if and only if there is a map $h' : N' \rightarrow M'$ such that $h' \circ f' = \textnormal{id}_{M'}$.  From the universal property of the pushout, the existence of such an $h'$ is equivalent to $g$ factoring through $f$.
\end{proof}

\begin{proposition}
\label{proposition-ML-characterization}
Let $M$ be an $R$-module.  Let $(M_i,f_{ij})$ be a directed system of finitely presented $R$-modules, indexed by $I$, such that $M = \colim M_i$.  Let $f_i: M_i \rightarrow M$ be the canonical map.  The following are equivalent:
\begin{enumerate}
\item For every finitely presented $R$-module $P$ and module map $f: P \rightarrow M$, there exists a finitely presented $R$-module $Q$ and a module map $g: P \rightarrow Q$ such that $g$ and $f$ dominate each other, i.e.\ $\ker(f \otimes_{R} \textnormal{id}_N) = \ker(g \otimes_{R} \textnormal{id}_N)$ for every $R$-module $N$.

\item For each $i \in I$, there exists $j \geq i$ such that $f_{ij}: M_i \rightarrow M_j$ dominates $f_i: M_i \rightarrow M$.

\item For each $i \in I$, there exists $j \geq i$ such that $f_{ij}: M_i \rightarrow M_j$ factors through $f_{ik}: M_i \rightarrow M_k$ for all $k \geq i$.

\item For every $R$-module $N$, the inverse system $(\Hom_{R}(M_i,N),\Hom_{R}(f_{ij},N))$ is Mittag-Leffler.

\item For $N = \prod_{s \in I} M_s$, the inverse system $(\Hom_{R}(M_i,N),\Hom_{R}(f_{ij},N))$ is Mittag-Leffler.
\end{enumerate}
\end{proposition}

\begin{proof}
First we prove the equivalence of (1) and (2).  Suppose (1) holds and let $i \in I$. Corresponding to the map $f_i: M_i \rightarrow M$, we can choose $g: M_i \rightarrow Q$ as in (1).  Since $M_i$ and $Q$ are of finite presentation, so is $\coker(g)$.  Then by Lemma \ref{lemma-domination}, $f_{i} : M_i \rightarrow M$ factors through $g: M_i \rightarrow Q$, say $f_i = h \circ g$ for some $h: Q \rightarrow M$.  Then since $Q$ is finitely presented, $h$ factors through $M_j \rightarrow M$ for some $j \geq i$, say $h = f_j \circ h'$ for some $h': Q \rightarrow M_j$.  In total we have a commutative diagram
\[ 
\xymatrix{
   										  &  M  & 			  \\
M_i \ar[dr]_{g} \ar[ur]^{f_i} \ar[rr]^{f_{ij}} &     & M_j \ar[ul]_{f_j}~ .\\
   										  & Q  \ar[ur]_{h'} &    
}
\]
Thus $f_{ij}$ dominates $g$.  But $g$ dominates $f_i$, so $f_{ij}$ dominates $f_i$.  

Conversely, suppose (2) holds.  Let $P$ be of finite presentation and $f: P \rightarrow M$ a module map.  Then $f$ factors through $f_i: M_i \rightarrow M$ for some $i \in I$, say $f = f_i \circ g'$ for some $g': P \rightarrow M_i$.  Choose by (2) a $j \geq i$ such that $f_{ij}$ dominates $f_i$.  We have a commutative diagram
\[
\xymatrix{
P \ar[d]_{g'} \ar[r]^{f}           & M  \\
M_i \ar[ur]^{f_i} \ar[r]_{f_{ij}} & M_j \ar[u]_{f_j} ~. 
} 
\]
From the diagram and the fact that $f_{ij}$ dominates $f_{i}$, we find that $f$ and $f_{ij} \circ g'$ dominate each other.  Hence taking $g = f_{ij} \circ g' : P \rightarrow M_j$ works.

Next we prove (2) is equivalent to (3).  Let $i \in I$.  It is always true that $f_{i}$ dominates $f_{ik}$ for $k \geq i$, since $f_{i}$ factors through $f_{ik}$.  If (2) holds, choose $j \geq i$ such that $f_{ij}$ dominates $f_{i}$.  Then since domination is a transitive relation, $f_{ij}$ dominates $f_{ik}$ for $k \geq i$. All $M_i$ are of finite presentation, so $\coker(f_{ik})$ is of finite presentation for $k \geq i$.  By Lemma \ref{lemma-domination}, $f_{ij}$ factors through $f_{ik}$ for all $k \geq i$.  Thus (2) implies (3).  On the other hand, if (3) holds then for any $R$-module $N$, $f_{ij} \otimes_{R} \textnormal{id}_{N}$ factors through $f_{ik} \otimes_{R} \textnormal{id}_N$ for $k \geq i$.  So $\ker(f_{ik} \otimes_{R} \textnormal{id}_N) \subset \ker(f_{ij} \otimes_{R} \textnormal{id}_{N})$ for $k \geq i$.  But $\ker(f_i \otimes_{R} \textnormal{id}_{N}: M_i \otimes_{R} N \rightarrow M \otimes_{R} N)$ is the union of $\ker(f_{ik} \otimes_{R} \textnormal{id}_N)$ for $k \geq i$.  Thus $\ker(f_i \otimes_{R} \textnormal{id}_{N}) \subset \ker(f_{ij} \otimes_{R} \textnormal{id}_{N})$ for any $R$-module $N$, which by definition means $f_{ij}$ dominates $f_{i}$.

It is trivial that (3) implies (4) implies (5).  We show (5) implies (3).  Let $N = \prod_{s \in I} M_s$. If (5) holds, then given $i \in I$ choose $j \geq i$ such that
\[ \text{Im}( \Hom(M_j, N) \rightarrow  \Hom(M_i, N)) =  \text{Im}( \Hom(M_k, N) \rightarrow  \Hom(M_i, N)) \]
for all $k \geq j$.  Passing the product over $s \in I$ outside of the $\Hom$'s and looking at the maps on each component of the product, this says
\[  
\text{Im}( \Hom(M_j, M_s) \rightarrow  \Hom(M_i, M_s)) =  \text{Im}( \Hom(M_k, M_s) \rightarrow   \Hom(M_i, M_s))
\]
for all $k \geq j$ and $s \in I$.  Taking $s = j$ we have
\[  
\text{Im}( \Hom(M_j, M_j) \rightarrow  \Hom(M_i, M_j)) =  \text{Im}( \Hom(M_k, M_j) \rightarrow   \Hom(M_i, M_j))
\]
for all $k \geq j$.  Since $f_{ij}$ is the image of $\textnormal{id} \in \Hom(M_j,M_j)$ under $\Hom(M_j, M_j) \rightarrow  \Hom(M_i, M_j)$, this shows that for any $k \geq j$ there is $h \in \Hom(M_k,M_j)$ such that $f_{ij} = h \circ f_{ik}$.  If $j \geq k$ then we can take $h = f_{kj}$.  Hence (3) holds.
\end{proof}

\begin{definition}
Let $M$ be an $R$-module.  We say that $M$ is \emph{Mittag-Leffler} if the equivalent conditions of Proposition \ref{proposition-ML-characterization} hold.
\end{definition}

\begin{remark}
\label{remark-flat-ML}
Let $M$ be a flat $R$-module.  By Lazard's theorem (Theorem \ref{theorem-lazard}) we can write $M = \colim M_i$ where the $M_i$ are free finite $R$-modules.  For $M$ to be Mittag-Leffler, it is enough for the inverse system of duals $(\Hom_{R}(M_i,R),\Hom_{R}(f_{ij},R))$ to be Mittag-Leffler.  This follows from criterion (4) of Proposition \ref{proposition-ML-characterization} and the fact that for a free finite $R$-module $F$, there is a functorial isomorphism $\Hom_{R}(F,R) \otimes_{R} N \cong \Hom_{R}(F,N)$ for any $R$-module $N$.
\end{remark}

\subsection{Interchanging direct products with tensor}
\noindent
Let $M$ be an $R$-module and let $(Q_{\alpha})_{\alpha \in A}$ be a family of $R$-modules.  Then there is a canonical map $M \otimes_{R} \left( \prod_{\alpha \in A} Q_{\alpha} \right) \rightarrow \prod_{\alpha \in A} ( M \otimes_{R} Q_{\alpha})$ given on pure tensors by $x \otimes (q_{\alpha}) \mapsto (x \otimes q_{\alpha})$.  This map is not necessarily injective or surjective, as the following example shows.

\begin{example}
\label{example-Q-not-ML}
Take $R = \mathbf{Z}$, $M = \mathbf{Q}$, and consider the family $Q_n = \mathbf{Z}/n$ for $n \geq 1$.  Then $\prod_{n} (M \otimes Q_n) = 0$.  However there is an injection $\mathbf{Q} \rightarrow M \otimes (\prod_{n} Q_n)$ obtained by tensoring the injection $\mathbf{Z} \rightarrow \prod_{n} Q_n$ by $M$, so $M \otimes (\prod_{n} Q_n)$ is nonzero.  Thus $M \otimes (\prod_{n} Q_n) \rightarrow \prod_{n} (M \otimes Q_n)$ is not injective.

On the other hand, take again $R = \mathbf{Z}$, $M = \mathbf{Q}$, and let $Q_n = \mathbf{Z}$ for $n \geq 1$.  The image of $M \otimes (\prod_{n} Q_n) \rightarrow \prod_{n} (M \otimes Q_n) = \prod_{n} M$ consists precisely of sequences of the form $(a_n/m)_{n \geq 1}$ with $a_n \in \mathbf{Z}$ and $m$ some nonzero integer.  Hence the map is not surjective.
\end{example}

\noindent
We determine below the precise conditions needed on $M$ for the map $M \otimes_{R} \left( \prod_{\alpha} Q_{\alpha} \right) \rightarrow \prod_{\alpha} (M \otimes_{R} Q_{\alpha})$ to be surjective, bijective, or injective for all choices of $(Q_{\alpha})_{\alpha \in A}$.  This is relevant because the modules for which it is injective turn out to be exactly Mittag-Leffler modules (Proposition \ref{proposition-ML-tensor}).  In what follows, if $M$ is an $R$-module and $A$ a set, we write $M^A$ for the product $\prod_{\alpha \in A} M$.

\begin{proposition}
\label{proposition-fg-tensor}
Let $M$ be an $R$-module.  The following are equivalent:
\begin{enumerate}
\item $M$ is finitely generated.
\item For every family $(Q_{\alpha})_{\alpha \in A}$ of $R$-modules, the canonical map $M \otimes_{R} \left( \prod_{\alpha} Q_{\alpha} \right) \rightarrow \prod_{\alpha} (M \otimes_{R} Q_{\alpha})$ is surjective.
\item For every $R$-module $Q$ and every set $A$, the canonical map $M \otimes_{R} Q^{A} \rightarrow (M \otimes_{R} Q)^{A}$ is surjective.
\item For every set $A$, the canonical map $M \otimes_{R} R^{A} \rightarrow M^{A}$ is surjective.
\end{enumerate}
\end{proposition}

\begin{proof}
First we prove (1) implies (2).  Choose a surjection $R^n \rightarrow M$ and consider the commutative diagram
\[
\xymatrix{
R^n \otimes_{R} (\prod_{\alpha} Q_{\alpha})  \ar[r]^{\cong} \ar[d] & \prod_{\alpha} (R^n \otimes_{R} Q_{\alpha}) \ar[d] \\
M \otimes_{R} (\prod_{\alpha} Q_{\alpha})  \ar[r] & \prod_{\alpha} ( M \otimes_{R} Q_{\alpha}).
}
\]
The top arrow is an isomorphism and the vertical arrows are surjections.  We conclude that the bottom arrow is a surjection.

Obviously (2) implies (3) implies (4), so it remains to prove (4) implies (1).  In fact for (1) to hold it suffices that the element $d = (x)_{x \in M}$ of $M^M$ is in the image of the map $f: M \otimes_{R} R^{M} \rightarrow M^M$.  In this case $d = \sum_{i = 1}^{n} f(x_i \otimes a_i)$ for some $x_i \in M$ and $a_i \in R^M$.  If for $x \in M$ we write $p_x: M^M \rightarrow M$ for the projection onto the $x$-th factor, then
\[ x = p_x(d) = \sum_{i = 1}^{n} p_x(f(x_i \otimes a_i)) = \sum_{i=1}^{n} p_x(a_i) x_i. \]
Thus $x_1, \dots, x_n$ generate $M$.
\end{proof}

\begin{proposition}
\label{proposition-fp-tensor}
Let $M$ be an $R$-module.  The following are equivalent:
\begin{enumerate}
\item $M$ is finitely presented.
\item For every family $(Q_{\alpha})_{\alpha \in A}$ of $R$-modules, the canonical map $M \otimes_{R} \left( \prod_{\alpha} Q_{\alpha} \right) \rightarrow \prod_{\alpha} (M \otimes_{R} Q_{\alpha})$ is bijective.
\item For every $R$-module $Q$ and every set $A$, the canonical map $M \otimes_{R} Q^{A} \rightarrow (M \otimes_{R} Q)^{A}$ is bijective.
\item For every set $A$, the canonical map $M \otimes_{R} R^{A} \rightarrow M^{A}$ is bijective.
\end{enumerate}
\end{proposition}

\begin{proof}
First we prove (1) implies (2).  Choose a presentation $R^m \rightarrow R^n \rightarrow M$ and consider the commutative diagram
\[
\xymatrix{
R^m \otimes_{R} (\prod_{\alpha} Q_{\alpha}) \ar[r] \ar[d]^{\cong} & R^m \otimes_{R} (\prod_{\alpha} Q_{\alpha}) \ar[r] \ar[d]^{\cong} & M \otimes_{R} (\prod_{\alpha} Q_{\alpha}) \ar[r] \ar[d] & 0 \\
\prod_{\alpha} (R^m \otimes_{R} Q_{\alpha}) \ar[r] & \prod_{\alpha} (R^n \otimes_{R}Q_{\alpha}) \ar[r] & \prod_{\alpha} (M \otimes_{R} Q_{\alpha}) \ar[r] & 0.
}
\]
The first two vertical arrows are isomorphisms and the rows are exact.  This implies that the map $M \otimes_{R} (\prod_{\alpha} Q_{\alpha})  \rightarrow \prod_{\alpha} ( M \otimes_{R} Q_{\alpha})$ is surjective and, by a diagram chase, also injective.  Hence (2) holds.

Obviously (2) implies (3) implies (4), so it remains to prove (4) implies (1).  From Proposition \ref{proposition-fg-tensor}, if (4) holds we already know that $M$ is finitely generated.  So we can choose a surjection $F \rightarrow M$ where $F$ is free and finite.  Let $K$ be the kernel.  We must show $K$ is finitely generated.  For any set $A$, we have a commutative diagram
\[
\xymatrix{
& K \otimes_{R} R^A \ar[r] \ar[d]_{f_3} & F \otimes_{R} R^A \ar[r] \ar[d]_{f_2}^{\cong} & M \otimes_{R} R^A \ar[r] \ar[d]_{f_1}^{\cong} & 0 \\
0 \ar[r] & K^A \ar[r] & F^A \ar[r] & M^A \ar[r] & 0 .
}
\]
The map $f_1$ is an isomorphism by assumption, the map $f_2$ is a isomorphism since $F$ is free and finite, and the rows are exact.  A diagram chase shows that $f_3$ is surjective, hence by Proposition \ref{proposition-fg-tensor} we get that $K$ is finitely generated. 
\end{proof}

\noindent
We need the following lemma for the next proposition.
\begin{lemma}
\label{lemma-kernel-tensored-fp}
Let $M$ be an $R$-module, $P$ a finitely presented $R$-module, and $f: P \rightarrow M$ a map.  Let $Q$ be an $R$-module and suppose $x \in \ker(P \otimes Q \rightarrow M \otimes Q)$.  Then there exists a finitely presented $R$-module $P'$ and a map $f': P \rightarrow P'$ such that $f$ factors through $f'$ and $x \in \ker(P \otimes Q \rightarrow P' \otimes Q)$.  
\end{lemma}

\begin{proof}
Write $M$ as a colimit $M = \colim_{i \in I} M_i$ of a directed system of finitely presented modules $M_i$.  Since $P$ is finitely presented, the map $f: P \rightarrow M$ factors through $M_j \rightarrow M$ for some $j \in I$.  Upon tensoring by $Q$ we have a commutative diagram
\[ 
\xymatrix{
& M_j \otimes Q \ar[dr] & \\
P \otimes Q \ar[ur] \ar[rr] & & M \otimes Q .
}
\]
The image $y$ of $x$ in $M_j \otimes Q$ is in the kernel of $M_j \otimes Q \rightarrow M \otimes Q$.  Since $M \otimes Q = \colim_{i \in I} (M_i \otimes Q)$, this means $y$ maps to $0$ in $M_{j'} \otimes Q$ for some $j' \geq j$.  Thus we may take $P' = M_{j'}$ and $f'$ to be the composite $P \rightarrow M_j \rightarrow M_{j'}$.
\end{proof}

\begin{proposition}
\label{proposition-ML-tensor}
Let $M$ be an $R$-module.  The following are equivalent:
\begin{enumerate}
\item $M$ is Mittag-Leffler.
\item For every family $(Q_{\alpha})_{\alpha \in A}$ of $R$-modules, the canonical map $M \otimes_{R} \left( \prod_{\alpha} Q_{\alpha} \right) \rightarrow \prod_{\alpha} (M \otimes_{R} Q_{\alpha})$ is injective.
\end{enumerate}
\end{proposition}

\begin{proof}
First we prove (1) implies (2).  Suppose $M$ is Mittag-Leffler and let $x$ be in the kernel of $M \otimes_{R} (\prod_{\alpha} Q_{\alpha}) \rightarrow \prod_{\alpha} (M \otimes_{R} Q_{\alpha})$.  Write $M$ as a colimit $M = \colim_{i \in I} M_i$ of a directed system of finitely presented modules $M_i$.  Then $M \otimes_{R} (\prod_{\alpha} Q_{\alpha})$ is the colimit of $M_i \otimes_{R} (\prod_{\alpha} Q_{\alpha})$.  So $x$ is the image of an element $x_i \in M_i \otimes_{R} (\prod_{\alpha} Q_{\alpha})$.  We must show that $x_i$ maps to $0$ in $M_j \otimes_{R} (\prod_{\alpha} Q_{\alpha})$ for some $j \geq i$.  Since $M$ is Mittag-Leffler, we may choose $j \geq i$ such that $M_{i} \rightarrow M_j$ and $M_i \rightarrow M$ dominate each other.  Then consider the commutative diagram
\[
\xymatrix{
M \otimes_{R} (\prod_{\alpha} Q_{\alpha}) \ar[r] & \prod_{\alpha} (M \otimes_{R} Q_{\alpha}) \\
M_i \otimes_{R} (\prod_{\alpha} Q_{\alpha}) \ar[r]^{\cong} \ar[d] \ar[u] & \prod_{\alpha} (M_i \otimes_{R} Q_{\alpha}) \ar[d] \ar[u] \\
M_j \otimes_{R} (\prod_{\alpha} Q_{\alpha}) \ar[r]^{\cong} & \prod_{\alpha} (M_j \otimes_{R} Q_{\alpha})
}
\]
whose bottom two horizontal maps are isomorphisms, according to Proposition \ref{proposition-fp-tensor}.  Since $x_i$ maps to $0$ in $\prod_{\alpha} (M \otimes_{R} Q_{\alpha})$, its image in $\prod_{\alpha} (M_i \otimes_{R} Q_{\alpha})$ is in the kernel of the map $\prod_{\alpha} (M_i \otimes_{R} Q_{\alpha}) \rightarrow \prod_{\alpha} (M \otimes_{R} Q_{\alpha})$.  But this kernel equals the kernel of $\prod_{\alpha} (M_i \otimes_{R} Q_{\alpha}) \rightarrow \prod_{\alpha} (M_j \otimes_{R} Q_{\alpha})$ according to the choice of $j$.  Thus $x_i$ maps to $0$ in $\prod_{\alpha} (M_j \otimes_{R} Q_{\alpha})$ and hence to $0$ in $M_j \otimes_{R} (\prod_{\alpha} Q_{\alpha})$.

Now suppose (2) holds. We prove $M$ satisfies formulation (1) of being Mittag-Leffler from Proposition \ref{proposition-ML-characterization}.  Let $f: P \rightarrow M$ be a map from a finitely presented module $P$ to $M$.  Choose a set $B$ of representatives of the isomorphism classes of finitely presented $R$-modules. Let $A$ be the set of pairs $(Q,x)$ where $Q \in B$ and $x \in \ker(P \otimes Q \rightarrow M \otimes Q)$.  For $\alpha = (Q,x) \in A$, we write $Q_{\alpha}$ for $Q$ and $x_{\alpha}$ for $x$.  Consider the commutative diagram
\[
\xymatrix{
M \otimes_{R} (\prod_{\alpha} Q_{\alpha}) \ar[r] & \prod_{\alpha} (M \otimes_{R} Q_{\alpha}) \\
P \otimes_{R} (\prod_{\alpha} Q_{\alpha}) \ar[r]^{\cong} \ar[u] & \prod_{\alpha} (P \otimes_{R} Q_{\alpha}) \ar[u] .
}
\]
The top arrow is an injection by assumption, and the bottom arrow is an isomorphism by Proposition \ref{proposition-fp-tensor}.  Let $x \in P \otimes_{R} (\prod_{\alpha} Q_{\alpha})$ be the element corresponding to $(x_{\alpha}) \in \prod_{\alpha} (P \otimes_{R} Q_{\alpha})$ under this isomorphism.  Then $x \in \ker( P \otimes_{R} (\prod_{\alpha} Q_{\alpha}) \rightarrow M \otimes_{R} (\prod_{\alpha} Q_{\alpha}))$ since the top arrow in the diagram is injective.  By Lemma \ref{lemma-kernel-tensored-fp}, we get a finitely presented module $P'$ and a map $f': P \rightarrow P'$ such that $f: P \rightarrow M$ factors through $f'$ and $x \in \ker(P \otimes_{R} (\prod_{\alpha} Q_{\alpha}) \rightarrow P' \otimes_{R} (\prod_{\alpha} Q_{\alpha}))$.  We have a commutative diagram
\[
\xymatrix{
P' \otimes_{R} (\prod_{\alpha} Q_{\alpha}) \ar[r]^{\cong} & \prod_{\alpha} (P' \otimes_{R} Q_{\alpha}) \\
P \otimes_{R} (\prod_{\alpha} Q_{\alpha}) \ar[r]^{\cong} \ar[u] & \prod_{\alpha} (P \otimes_{R} Q_{\alpha}) \ar[u] .
}
\]
where both the top and bottom arrows are isomorphisms by Proposition \ref{proposition-fp-tensor}.  Thus since $x$ is in the kernel of the left vertical map, $(x_{\alpha})$ is in the kernel of the right vertical map.  This means $x_{\alpha} \in \ker(P \otimes_{R} Q_{\alpha} \rightarrow P' \otimes_{R} Q_{\alpha})$ for every $\alpha \in A$.  By the definition of $A$ this means $\ker(P \otimes_{R} Q \rightarrow P' \otimes_{R} Q) \subset \ker(P \otimes_{R} Q \rightarrow M \otimes_{R} Q)$ for all finitely presented $Q$ and, since $f: P \rightarrow M$ factors through $f': P \rightarrow P'$, actually equality holds.  By Lemma \ref{lemma-domination-fp}, $f$ and $f'$ dominate each other.
\end{proof}

\begin{lemma}
\label{lemma-pure-submodule-ML}
Let $0 \rightarrow M_1 \rightarrow M_2 \rightarrow M_3 \rightarrow 0$ be a universally exact sequence of $R$-modules.  Then:
\begin{enumerate}
\item If $M_2$ is Mittag-Leffler, then $M_1$ is Mittag-Leffler.
\item If $M_1$ and $M_3$ are Mittag-Leffler, then $M_2$ is Mittag-Leffler.
\end{enumerate}
\end{lemma}

\begin{proof}
For any family $(Q_{\alpha})_{\alpha \in A}$ of $R$-modules we have a commutative diagram
\[
\xymatrix{
0 \ar[r] & M_1 \otimes_{R} (\prod_{\alpha} Q_{\alpha}) \ar[r] \ar[d] & M_2 \otimes_{R} (\prod_{\alpha} Q_{\alpha}) \ar[r] \ar[d] & M_3 \otimes_{R} (\prod_{\alpha} Q_{\alpha}) \ar[r] \ar[d] & 0 \\
0 \ar[r] & \prod_{\alpha}(M_1 \otimes Q_{\alpha}) \ar[r] & \prod_{\alpha}(M_2 \otimes Q_{\alpha}) \ar[r] & \prod_{\alpha}(M_3 \otimes Q_{\alpha})\ar[r] & 0
}
\]
with exact rows. Thus (1) and (2) follow from Proposition \ref{proposition-ML-tensor}.
\end{proof}

\begin{lemma}
\label{lemma-direct-sum-ML}
If $M = \bigoplus_{i \in I} M_{i}$ is a direct sum of $R$-modules, then $M$ is Mittag-Leffler if and only if each $M_i$ is Mittag-Leffler.
\end{lemma}

\begin{proof}
The ``only if'' direction follows from Lemma \ref{lemma-pure-submodule-ML} (1) and the fact that a split short exact sequence is universally exact.  For the converse, first note that if $I$ is finite then this follows from Lemma \ref{lemma-pure-submodule-ML} (2).  For general $I$, if all $M_i$ are Mittag-Leffler then we prove the same of $M$ by verifying condition (1) of Proposition \ref{proposition-ML-characterization}.  Let $f: P \rightarrow M$ be a map from a finitely presented module $P$.  Then $f$ factors as $P \xrightarrow{f'} \bigoplus_{i' \in I'} M_{i'} \hookrightarrow \bigoplus_{i \in I} M_i$ for some finite subset $I'$ of $I$.  By the finite case $\bigoplus_{i' \in I'} M_{i'}$ is Mittag-Leffler and hence there exists a finitely presented module $Q$ and a map $g: P \rightarrow Q$ such that $g$ and $f'$ dominate each other.  Then also $g$ and $f$ dominate each other.
\end{proof}

\subsection{Examples}
\label{section-examples} 
We end this section with some examples and non-examples of Mittag-Leffler modules.

\begin{examples} 
\label{examples-ML}
\begin{enumerate}
\item Any finitely presented module is Mittag-Leffler.  This follows, for instance, from Proposition \ref{proposition-ML-characterization} (1).  In general, it is true that a finitely generated module is Mittag-Leffler if and only it is finitely presented.  This follows from Propositions \ref{proposition-fg-tensor}, \ref{proposition-fp-tensor}, and \ref{proposition-ML-tensor}.

\item A free module is Mittag-Leffler since it satisfies condition (1) of Proposition \ref{proposition-ML-characterization}.

\item By the previous example together with Lemma \ref{lemma-direct-sum-ML}, projective modules are Mittag-Leffler. 
\end{enumerate}
\end{examples}

\noindent
We also want to add to our list of examples power series rings over a Noetherian ring $R$.  This will be a consequence the following lemma.
\begin{lemma}
\label{lemma-flat-ML-criterion}
Let $M$ be a flat $R$-module.  Suppose the following condition holds: if $F$ is a free finite $R$-module and $x \in F \otimes_{R} M$, then there exists a smallest submodule $F'$ of $F$ such that $x \in F' \otimes_{R} M$.  Then $M$ is Mittag-Leffler.
\end{lemma}

\begin{proof}
By Theorem \ref{theorem-lazard} we can write $M$ as the colimit $M = \colim_{i \in I} M_i$ of a directed system $(M_i, f_{ij})$ of free finite $R$-modules.  By Remark \ref{remark-flat-ML}, it suffices to show that the inverse system $(\Hom_{R}(M_i, R), \Hom_{R}(f_{ij},R))$ is Mittag-Leffler.  In other words, fix $i \in I$ and for $j \geq i$ let $Q_j$ be the image of $\Hom_{R}(M_j,R) \rightarrow \Hom_{R}(M_i,R)$; we must show that the $Q_{j}$ stabilize.  

Since $M_i$ is free and finite, we can make the identification $\Hom_{R}(M_i,M_j) = \Hom_{R}(M_i,R) \otimes_{R}  M_j$ for all $j$.  Using the fact that the $M_j$ are free, it follows that for $j \geq i$, $Q_j$ is the smallest submodule of $\Hom_{R}(M_i,R)$ such that $f_{ij} \in Q_j \otimes_{R} M_j$.  Under the identification $\Hom_{R}(M_i,M) = \Hom_{R}(M_i,R) \otimes_{R} M$, the canonical map $f_i: M_i \rightarrow M$ is in $\Hom_{R}(M_i,R) \otimes_{R} M$.  By the assumption on $M$, there exists a smallest submodule $Q$ of $\Hom_{R}(M_i,R)$ such that $f_i \in Q \otimes_{R} M$.  We are going to show that the $Q_j$ stabilize to $Q$.

For $j \geq i$ we have a commutative diagram
\[ 
\xymatrix{
Q_j \otimes_{R} M_j \ar[r] \ar[d] & \Hom_{R}(M_i,R) \otimes_{R} M_j \ar[d] \\
Q_j \otimes_{R} M \ar[r] & \Hom_{R}(M_i,R) \otimes_{R} M.
}
\]
Since $f_{ij} \in Q_j \otimes_{R} M_j$ maps to $f_i \in \Hom_{R}(M_i,R) \otimes_{R} M$, it follows that $f_i \in Q_j \otimes_{R} M$.  Hence, by the choice of $Q$, we have $Q \subset Q_j$ for all $j \geq i$.

Since the $Q_j$ are decreasing and $Q \subset Q_j$ for all $j \geq i$, to show that the $Q_j$ stabilize to $Q$ it suffices to find a $j \geq i$ such that $Q_j \subset Q$.  As an element of
\[ 
\Hom_{R}(M_i,R) \otimes_{R} M = \colim_{j \in J} (\Hom_{R}(M_i,R) \otimes_{R} M_j),
\]
$f_i$ is the colimit of $f_{ij}$ for $j \geq i$, and $f_i$ also lies in the submodule
\[ 
\colim_{j \in J} (Q \otimes_{R} M_j) \subset \colim_{j \in J} (\Hom_{R}(M_i,R) \otimes_{R} M_j).
\]
It follows that for some $j \geq i$, $f_{ij}$ lies in $Q \otimes_{R} M_j$.  Since $Q_{j}$ is the smallest submodule of $\Hom_{R}(M_i,R)$ with $f_{ij} \in Q_j \otimes_{R} M_j$, we conclude $Q_j\subset Q$.
\end{proof}

\begin{lemma}
\label{lemma-power-series-ML}
Let $R$ be a Noetherian ring and $n$ a positive integer.  Then the $R$-module $M = R[[t_1, \dots, t_n]]$ is flat and Mittag-Leffler.
\end{lemma}

\begin{proof}
Since $M$ is the completion of the Noetherian ring $R[t_1,\dots,t_n]$ with respect to the ideal $(t_1, \dots, t_n)$, $M$ is flat over $R[t_1, \dots, t_n]$.  Thus since $R[t_1, \dots, t_n]$ is flat over $R$, so is $M$.  We show that $M$ satisfies the condition of Lemma \ref{lemma-flat-ML-criterion}.  Let $F$ be a free finite $R$-module.  As an $R$-module, we make the identification $M = R^{I}$ for a (countable) set $I$.  If $F'$ is any submodule of $F$ then it is finitely presented since $R$ is Noetherian.  So by Proposition \ref{proposition-fp-tensor} we have a commutative diagram
\[
\xymatrix{
F' \otimes_{R} M \ar[r] \ar[d]^{\cong} & F \otimes_{R} M \ar[d]^{\cong} \\
(F')^I \ar[r] & F^I
}
\]
by which we can identify the map $F' \otimes_{R} M \rightarrow F \otimes_{R} M$ with $(F')^I \rightarrow F^I$.  Hence if $x \in F \otimes_{R} M$ corresponds to $(x_i) \in F^I$, then the submodule of $F'$ of $F$ generated by the $x_i$ is the smallest submodule of $F$ such that $x \in F' \otimes_{R} M$.
\end{proof}

\begin{non-examples}
\label{non-examples-ML}
\begin{enumerate}
\item By Example \ref{example-Q-not-ML} and Proposition \ref{proposition-ML-tensor}, $\mathbf{Q}$ is not a Mittag-Leffler $\mathbf{Z}$-module.

\item We prove below (Theorem \ref{theorem-projectivity-characterization}) that for a flat and countably generated module, projectivity is equivalent to being Mittag-Leffler.  Thus any flat, countably generated, non-projective module $M$ is an example of a non-Mittag-Leffler module.  For such an example, see \cite[Remark 7.65.3]{stacks-project}.
\end{enumerate}
\end{non-examples}

\section{A characterization of projective modules}
\noindent
The goal of this section is to prove that a module is projective if and only if it is flat, Mittag-Leffler, and a direct sum of countably generated modules (Theorem \ref{theorem-projectivity-characterization} below).

\begin{lemma}
\label{lemma-ML-countable-colimit}
Let $M$ be an $R$-module.  Write $M = \colim_{i \in I} M_i$ where $(M_i, f_{ij})$ is a directed system of finitely presented $R$-modules.  If $M$ is Mittag-Leffler and countably generated, then there is a directed countable subset $I' \subset I$ such that $M \cong \colim_{i \in I'} M_i$.
\end{lemma}

\begin{proof}
Let $x_1,x_2, \dots$ be a countable set of generators for $M$.  For each $x_n$ choose $i \in I$ such that $x_n$ is in the image of the canonical map $f_{i}: M_{i} \rightarrow M$; let $I'_{0} \subset I$ be the set of all these $i$.  Now since $M$ is Mittag-Leffler, for each $i \in I'_{0}$ we can choose $j \in I$ such that $j \geq i$ and $f_{ij}: M_{i} \rightarrow M_{j}$ factors through $f_{ik}: M_i \rightarrow M_k$ for all $k \geq i$  (condition (3) of Proposition \ref{proposition-ML-characterization}); let $I'_1$ be the union of $I'_0$ with all of these $j$.  Since $I'_1$ is a countable, we can enlarge it to a countable directed set $I'_{2} \subset I$.  Now we can apply the same procedure to $I'_{2}$ as we did to $I'_{0}$ to get a new countable set $I'_{3} \subset I$.  Then we enlarge $I'_{3}$ to a countable directed set $I'_{4}$.  Continuing in this way---adding in a $j$ as in Proposition \ref{proposition-ML-characterization} (3) for each $ i \in I'_{\ell}$ if $\ell$ is odd and enlarging $I'_{\ell}$ to a directed set if $\ell$ is even---we get a sequence of subsets $I'_{\ell} \subset I$ for $\ell \geq 0$.  The union $I' = \bigcup I'_{\ell}$ satisfies: 
\begin{enumerate}
\item $I'$ is countable and directed; 
\item each $x_{n}$ is in the image of $f_{i}: M_i \rightarrow M$ for some $i \in I'$;
\item if $i \in I'$, then there is $j \in I'$ such that $j \geq i$ and $f_{ij}: M_{i} \rightarrow M_{j}$ factors through $f_{ik}: M_i \rightarrow M_k$ for all $k \in I$ with $k \geq i$.  In particular $\ker(f_{ik}) \subset \ker(f_{ij})$ for $k \geq i$.
\end{enumerate}
We claim that the canonical map $\colim_{i \in I'} M_i \rightarrow \colim_{i \in I} M_i = M$ is an isomorphism.  By (2) it is surjective.  For injectivity, suppose $x \in \colim_{i \in I'} M_i$ maps to $0$ in $\colim_{i \in I} M_i$.  Representing $x$ by an element $\tilde{x} \in M_i$ for some $i \in I'$, this means that $f_{ik}(\tilde{x}) = 0$ for some $k \in I, k \geq i$.  But then by (3) there is $j \in I', j \geq i,$ such that $f_{ij}(\tilde{x}) = 0$.  Hence $x = 0$ in $\colim_{i \in I'} M_i$.
\end{proof}

\begin{lemma}
\label{lemma-countgen-projective}
Let $M$ be an $R$-module.  If $M$ is flat, Mittag-Leffler, and countably generated, then $M$ is projective.
\end{lemma}

\begin{proof}
By Lazard's theorem (Theorem \ref{theorem-lazard}), we can write $M = \colim_{i \in I} M_i$ for a directed system of finite free $R$-modules $(M_i, f_{ij})$ indexed by a set $I$.  By Lemma \ref{lemma-ML-countable-colimit}, we may assume $I$ is countable.  Now let
\[ 0 \rightarrow N_1 \rightarrow N_2 \rightarrow N_3 \rightarrow 0 \]
be an exact sequence of $R$-modules.  We must show that applying $\Hom_R(M,-)$ preserves exactness.  Since $M_i$ is finite free,
\[ 0 \rightarrow \Hom_R(M_i,N_1) \rightarrow \Hom_R(M_i,N_2) \rightarrow \Hom_R(M_i,N_3) \rightarrow 0 \]
is exact for each $i$.  Since $M$ is Mittag-Leffler, $(\Hom_R(M_i, N_{1}))$ is a Mittag-Leffler inverse system.  So by Lemma \ref{lemma-ML-exact-sequence}, 
\[ 0 \rightarrow \lim_{i \in I} \Hom_R(M_i,N_1) \rightarrow \lim_{i \in I} \Hom_R(M_i,N_2) \rightarrow \lim_{i \in I} \Hom_R(M_i,N_3) \rightarrow 0 \]
is exact.  But for any $R$-module $N$ there is a functorial isomorphism $\Hom_R(M,N) \cong \lim_{i \in I} \Hom_R(M_i,N)$, so
\[ 0 \rightarrow \Hom_R(M,N_1) \rightarrow \Hom_R(M,N_2) \rightarrow \Hom_R(M,N_3) \rightarrow 0 \]
is exact.
\end{proof}

\begin{remark}
Lemma \ref{lemma-countgen-projective} does not hold without the countable generation assumption.  For example, the $\mathbf Z$-module $M = \mathbf{Z}[[x]]$ is flat and Mittag-Leffler but not projective.  It is Mittag-Leffler by Lemma \ref{lemma-power-series-ML}.  Subgroups of free abelian groups are free, hence a projective $\mathbf Z$-module is in fact free and so are its submodules.  Thus to show $M$ is not projective it suffices to produce a non-free submodule.  Fix a prime $p$ and consider the submodule $N$ consisting of power series $f(x) = \sum a_i x^i$ such that for every integer $m \geq 1$, $p^m$ divides $a_i$ for all but finitely many $i$.  Then $\sum a_i p^i x^i$ is in $N$ for all $a_i \in \mathbf{Z}$, so $N$ is uncountable.  Thus if $N$ were free it would have uncountable rank and the dimension of $N/pN$ over $\mathbf{Z}/p$ would be uncountable.  This is not true as the elements $x^i \in N/pN$ for $i \geq 0$ span $N/pN$.
\end{remark}

\begin{theorem}
\label{theorem-projectivity-characterization}
Let $M$ be an $R$-module.  Then $M$ is projective if and only it satisfies:
\begin{enumerate}
\item $M$ is flat,
\item $M$ is Mittag-Leffler, 
\item $M$ is a direct sum of countably generated $R$-modules.
\end{enumerate}
\end{theorem}

\begin{proof}
First suppose $M$ is projective.  Then $M$ is a direct summand of a free module, so $M$ is flat and Mittag-Leffler since these properties pass to direct summands. By Kaplansky's theorem (Theorem \ref{theorem-projective-direct-sum}), $M$ satisfies (3).

Conversely, suppose $M$ satisfies (1)-(3).  Since being flat and Mittag-Leffler passes to direct summands, $M$ is a direct sum of flat, Mittag-Leffler, countably generated $R$-modules.  Thus by the previous lemma $M$ is a direct sum of projective modules.  Hence $M$ is projective.
\end{proof}

\begin{lemma}
\label{lemma-ML-ui-descent}
Let $f: M \rightarrow N$ be universally injective map of $R$-modules.  Suppose $M$ is a direct sum of countably generated $R$-modules, and suppose $N$ is flat and Mittag-Leffler.  Then $M$ is projective. 
\end{lemma}

\begin{proof}
By Lemmas \ref{lemma-ui-flat-domain} and \ref{lemma-pure-submodule-ML}, $M$ is flat and Mittag-Leffler, so the conclusion follows from Theorem \ref{theorem-projectivity-characterization}.
\end{proof}

\begin{lemma}
Let $R$ be a Noetherian ring and let $M$ be a $R$-module.  Suppose $M$ is a direct sum of countably generated $R$-modules, and suppose there is a universally injective map $M \rightarrow R[[t_1, \dots, t_n]]$ for some $n$.  Then $M$ is projective.  
\end{lemma}

\begin{proof}
Follows from the previous lemma and Lemma \ref{lemma-power-series-ML}.
\end{proof}

\section{Ascending properties}
\noindent
All of the properties of a module in Theorem \ref{theorem-projectivity-characterization} ascend along arbitrary ring maps:
\begin{lemma}
\label{lemma-ascend-properties-modules}
Let $R \rightarrow S$ be a ring map.  Let $M$ be an $R$-module.  Then:
\begin{enumerate}
\item If $M$ is flat, then the $S$-module $M \otimes_{R} S$ is flat.
\item If $M$ is Mittag-Leffler, then the $S$-module $M \otimes_{R} S$ is Mittag-Leffler.
\item If $M$ is a direct sum of countably generated $R$-modules, then the $S$-module $M \otimes_{R} S$ is a direct sum of countably generated $S$-modules.
\item If $M$ is projective, then the $S$-module $M \otimes_{R} S$ is projective.
\end{enumerate}
\end{lemma}

\begin{proof}
All are obvious except (2).  For this, use formulation (3) of being Mittag-Leffler from Proposition \ref{proposition-ML-characterization} and the fact that tensoring commutes with taking colimits.
\end{proof}

\section{Descending properties}
\noindent
We address the faithfully flat descent of the properties from Theorem \ref{theorem-projectivity-characterization} that characterize projectivity.

\begin{lemma}
\label{lemma-ffdescent-flatness}
Let $R \rightarrow S$ be a faithfully flat ring map.  Let $M$ be an $R$-module.  If the $S$-module $M \otimes_{R} S$ is flat, then $M$ is flat. 
\end{lemma}

\begin{proof}
This is \cite[Lemma 7.28.6]{stacks-project}.
\end{proof}

\noindent
In the presence of flatness, the property of being a Mittag-Leffler module descends:
\begin{lemma}
\label{lemma-ffdescent-flat-ML}
Let $R \rightarrow S$ be a faithfully flat ring map.  Let $M$ be an $R$-module.  If the $S$-module $M \otimes_{R} S$ is flat and Mittag-Leffler, then $M$ is flat and Mittag-Leffler. 
\end{lemma}

\begin{proof}
By the previous lemma, flatness descends, so $M$ is flat.  Thus by Lazard's theorem (Theorem \ref{theorem-lazard}) we can write $M = \colim_{i \in I} M_i$ where $(M_i,f_{ij})$ is a directed system of free finite $R$-modules.  According to Remark \ref{remark-flat-ML}, to prove $M$ is Mittag-Leffler it is enough to show that $(\Hom_{R}(M_i,R))$ is a Mittag-Leffler inverse system.

Since tensoring commutes with colimits, $M \otimes_{R} S = \colim (M_i \otimes_{R} S)$.  Since $M \otimes_{R} S$ is Mittag-Leffler this means $(\Hom_{S}(M_i \otimes_{R} S, S))$ is a Mittag-Leffler inverse system.  So for every $i \in I$, the family $\text{Im}(\Hom_{S}(M_j \otimes_{R} S, S) \rightarrow \Hom_{S}(M_i \otimes_{R} S, S))$ for $j \geq i$ stabilizes.  Because $M_i$ is free and finite there is a functorial isomorphism $\Hom_{S}(M_i \otimes_{R} S,S) \cong \Hom_{R}(M_i,R) \otimes_{R} S$, and because $R \rightarrow S$ is faithfully flat, tensoring by $S$ commutes with taking the image of a module map.  Thus we find that for every $i \in I$, the family $\text{Im}(\Hom_{R}(M_j, R) \rightarrow \Hom_{R}(M_i, R)) \otimes_{R} S$ for $j \geq i$ stabilizes.  But if $N$ is an $R$-module and $N' \subset N$ a submodule such that $N' \otimes_{R} S = N \otimes_{R} S$, then $N' = N$ by faithful flatness of $S$.  We conclude that for every $i \in I$, the family $\text{Im}(\Hom_{R}(M_j, R) \rightarrow \Hom_{R}(M_i, R))$ for $j \geq i$ stabilizes.  So $M$ is Mittag-Leffler.
\end{proof}

\noindent
At this point the faithfully flat descent of countably generated projective modules follows easily.
\begin{lemma}
\label{lemma-ffdescent-countable-projectivity}
Let $R \rightarrow S$ be a faithfully flat ring map.  Let $M$ be an $R$-module.  If the $S$-module $M \otimes_{R} S$ is countably generated and projective, then $M$ is countably generated and projective. 
\end{lemma}

\begin{proof}
Follows from Lemma \ref{lemma-ffdescent-flat-ML}, the fact that countable generation descends, and Theorem \ref{theorem-projectivity-characterization}.
\end{proof}

\noindent
All that remains is to use d\'evissage to reduce descent of projectivity in the general case to the countably generated case.  First, two simple lemmas.

\begin{lemma}
\label{lemma-lift-countably-generated-submodule}
Let $R \rightarrow S$ be a ring map, let $M$ be an $R$-module, and let $Q$ be a countably generated $S$-submodule of $M \otimes_{R} S$.  Then there exists a countably generated $R$-submodule $P$ of $M$ such that $\textnormal{Im}(P \otimes_{R} S \rightarrow M \otimes_{R} S)$ contains $Q$.
\end{lemma}

\begin{proof}
Let $y_1, y_2, \dots$ be generators for $Q$ and write $y_{j} = \sum_{k} x_{jk} \otimes s_{jk}$ for some $x_{jk} \in M$ and $s_{jk} \in S$. Then take $P$ be the submodule of $M$ generated by the $x_{jk}$.
\end{proof}

\begin{lemma}
\label{lemma-adapted-submodule}
Let $R \rightarrow S$ be a ring map, and let $M$ be an $R$-module.  Suppose $M \otimes_{R} S = \bigoplus_{i \in I} Q_i$ is a direct sum of countably generated $S$-modules $Q_i$.  If $N$ is a countably generated submodule of $M$, then there is a countably generated submodule $N'$ of $M$ such that $N' \supset N$ and $\textnormal{Im}(N' \otimes_{R} S \rightarrow M \otimes_{R} S) = \bigoplus_{i \in I'} Q_i$ for some subset $I' \subset I$.
\end{lemma}

\begin{proof}
Let $N'_0 = N$.  We construct by induction an increasing sequence of countably generated submodules $N'_{\ell} \subset M$ for $\ell = 0,1,2, \dots$ such that: if $I'_{\ell}$ is the set of $i \in I$ such that the projection of $\text{Im}(N'_{\ell} \otimes_{R} S \rightarrow M \otimes_{R} S)$ onto $Q_i$ is nonzero, then $\text{Im}(N'_{\ell+1} \otimes_{R} S \rightarrow M \otimes_{R} S)$ contains $Q_i$ for all $i \in I'_{\ell}$.  To construct $N'_{\ell+1}$ from $N'_{\ell}$, let $Q$ be the sum of (the countably many) $Q_{i}$ for $i \in I'_{\ell}$, choose $P$ as in Lemma \ref{lemma-lift-countably-generated-submodule}, and then let $N'_{\ell + 1} = N'_{\ell} + P$.  Having constructed the $N'_{\ell}$, just take $N' = \bigcup_{\ell} N'_{\ell}$ and $I' = \bigcup_{\ell} I'_{\ell}$.
\end{proof}

\begin{theorem}
\label{theorem-ffdescent-projectivity}
Let $R \rightarrow S$ be a faithfully flat ring map.  Let $M$ be an $R$-module.  If the $S$-module $M \otimes_{R} S$ is projective, then $M$ is projective. 
\end{theorem}

\begin{proof}
We are going to construct a Kaplansky d\'evissage of $M$ to show that it is a direct sum of projective modules and hence projective.  By Theorem \ref{theorem-projective-direct-sum} we can write $M \otimes_{R} S = \bigoplus_{i \in I} Q_i$ as a direct sum of countably generated $S$-modules $Q_i$.  Choose a well-ordering on $M$.  By transfinite induction we are going to define an increasing family of submodules $M_{\alpha}$ of $M$, one for each ordinal $\alpha$, such that $M_{\alpha} \otimes_{R} S$ is a direct sum of some subset of the $Q_i$.

For $\alpha = 0$ let $M_0 = 0$.  If $\alpha$ is a limit ordinal and $M_{\beta}$ has been defined for all $\beta < \alpha$, then define $M_{\beta} = \bigcup_{\beta < \alpha} M_{\beta}$.  Since each $M_{\beta} \otimes_{R} S$ for $\beta < \alpha$ is a direct sum of a subset of the $Q_i$, the same will be true of $M_{\alpha} \otimes_{R} S$.  If $\alpha+1$ is a successor ordinal and $M_{\alpha}$ has been defined, then define $M_{\alpha+1}$ as follows.  If $M_{\alpha} = M$, then let $M_{\alpha +1} = M$.  Otherwise choose the smallest $x \in M$ (with respect to the fixed well-ordering) such that $x \notin M_{\alpha}$. Since $S$ is flat over $R$, $(M/M_{\alpha}) \otimes_{R} S = M \otimes_{R} S/M_{\alpha} \otimes_{R} S$, so since $M_{\alpha} \otimes_{R} S$ is a direct sum of some $Q_{i}$, the same is true of $(M/M_{\alpha}) \otimes_{R} S$.   By Lemma \ref{lemma-adapted-submodule}, we can find a countably generated $R$-submodule $P$ of $M/M_{\alpha}$ containing the image of $x$ in $M/M_{\alpha}$ and such that $P \otimes_{R} S$ (which equals $\text{Im}(P \otimes_{R} S \rightarrow M \otimes_{R} S)$ since $S$ is flat over $R$) is a direct sum of some $Q_{i}$. Since $M \otimes_{R} S = \bigoplus_{i \in I} Q_i$ is projective and projectivity passes to direct summands, $P \otimes_{R} S$ is also projective.  Thus by Lemma \ref{lemma-ffdescent-countable-projectivity}, $P$ is projective.  Finally we define $M_{\alpha+1}$ to be the preimage of $P$ in $M$, so that $M_{\alpha+1}/M_{\alpha} = P$ is countably generated and projective.  In particular $M_{\alpha}$ is a direct summand of $M_{\alpha+1}$ since projectivity of $M_{\alpha+1}/M_{\alpha}$ implies the sequence $0 \rightarrow M_{\alpha} \rightarrow M_{\alpha+1} \rightarrow M_{\alpha+1}/M_{\alpha} \rightarrow 0$ splits.

Transfinite induction on $M$ (using the fact that we constructed $M_{\alpha+1}$ to contain the smallest $x \in M$ not contained in $M_{\alpha}$) shows that each $x \in M$ is contained in some $M_{\alpha}$.  Thus, there is some large enough ordinal $S$ satisfying: for each $x \in M$ there is $\alpha \in S$ such that $x \in M_{\alpha}$.  This means $(M_{\alpha})_{\alpha \in S}$ satisfies property (1) of a Kaplansky d\'evissage of $M$.  The other properties are clear by construction.  We conclude $M = \bigoplus_{\alpha+1 \in S} M_{\alpha+1}/M_{\alpha}$.  Since each $M_{\alpha+1}/M_{\alpha}$ is projective by construction, $M$ is projective.  
\end{proof}

\bibliographystyle{amsplain}
\bibliography{ffdescent}
\end{document}